\documentclass[11pt]{amsart}

\usepackage{epigamath}

\usepackage[english]{babel}

\numberwithin{equation}{section}

\usepackage{calc}
\usepackage{amssymb, amsthm, amsmath, amscd, amsfonts}
\usepackage[cal=dutchcal]{mathalpha} 
\usepackage{mathtools}
\usepackage{enumitem}
\usepackage{stmaryrd}  

\input{xy}
\xyoption{all}

\makeatletter
\@ifdefinable\equationname{\let\equationname\equationautorefname}
\def\equationautorefname~#1\@empty\@empty\null{(#1\@empty\@empty\null)}
\@ifdefinable\AMSname{\let\AMSname\AMSautorefname}
\def\AMSautorefname~#1\@empty\@empty\null{(#1\@empty\@empty\null)}
\@ifdefinable\itemname{\let\itemname\itemautorefname}
\def\itemautorefname~#1\@empty\@empty\null{(#1\@empty\@empty\null)}

\renewcommand{\theenumi}{\arabic{enumi}}

\renewcommand{\theenumii}{\roman{enumii}}

\renewcommand{\p@enumii}{\theenumi$\m@th\vert$}

\renewcommand{\p@enumiii}{\theenumi.\theenumii.}

\renewcommand{\labelitemi}{$\m@th\bullet$}
\renewcommand{\labelitemii}{$\m@th\diamond$}
\renewcommand{\labelitemiii}{$\m@th\star$}
\renewcommand{\labelitemiv}{$\m@th\cdot$}
\makeatother

\RequirePackage{aliascnt}
\newcommand{\basetheorem}[3]{%
    \newtheorem{#1}{#2}[#3]
    \newtheorem*{#1*}{#2}
    \expandafter\def\csname #1autorefname\endcsname{#2}
}%
\newcommand{\maketheorem}[3]{%
    \newaliascnt{#1}{#3}
    \newtheorem{#1}[#1]{#2}
    \aliascntresetthe{#1}
    \expandafter\def\csname #1autorefname\endcsname{#2}
    \newtheorem{#1*}{#2}
}%
\basetheorem{theorem}{Theorem}{section}
\maketheorem{proposition}{Proposition}{theorem}
\maketheorem{corollary}{Corollary}{theorem}
\maketheorem{lemma}{Lemma}{theorem}
\maketheorem{conjecture}{Conjecture}{theorem}
\maketheorem{scholium}{Scholium}{theorem}
\maketheorem{claim}{Claim}{theorem}

\newtheorem{definition-proposition}[theorem]{Definition-Proposition}
\newtheorem*{mainthm*}{Main Theorem}
\newtheorem*{theoremA*}{Theorem A}
\newtheorem*{theoremB*}{Theorem B}
\newtheorem*{theoremC*}{Theorem C}
\newtheorem*{theoremD*}{Theorem D}
\newtheorem*{theoremE*}{Theorem E}
\newtheorem*{theoremF*}{Theorem F}
\newtheorem*{theoremG*}{Theorem G}
\newtheorem*{keylemma*}{Key Lemma}

\theoremstyle{definition}

\maketheorem{definition}{Definition}{theorem}
\maketheorem{example}{Example}{theorem}
\maketheorem{remark}{Remark}{theorem}
\maketheorem{question}{Question}{theorem}
\maketheorem{convention}{Convention}{theorem}
\maketheorem{setup}{Setup}{theorem}
\maketheorem{terminology}{Terminology}{theorem}

\newcommand{\N}{\mathbb N}
\newcommand{\Z}{\mathbb Z}
\newcommand{\Q}{\mathbb Q}

\newcommand{\bP}{\mathbb P}

\newcommand{\aaa}{\mathfrak{a}}
\newcommand{\bbb}{\mathfrak{b}}
\newcommand{\mmm}{\mathfrak{m}}
\newcommand{\nnn}{\mathfrak{n}}
\newcommand{\p}{\mathfrak{p}}

\newcommand{\sA}{\mathcal{A}}

\newcommand{\sF}{\mathcal{F}}

\newcommand{\sL}{\mathcal{L}}
\newcommand{\sO}{\mathcal{O}}

\newcommand{\kay}{\mathcal{k}}
\newcommand{\el}{\mathcal{l}}

\DeclareMathOperator{\sHom}{\mathcal{H}\hspace{-.2ex}\mathcal{o}\hspace{-.1ex}\mathcal{m}}

\DeclareMathOperator{\Aut}{Aut}

\DeclareMathOperator{\Cl}{Cl}

\DeclareMathOperator{\coker}{{coker}}

\DeclareMathOperator{\disc}{disc}
\DeclareMathOperator{\divisor}{{div}}

\DeclareMathOperator{\id}{{id}}

\DeclareMathOperator{\Spec}{Spec}
\DeclareMathOperator{\Tr}{Tr}

\DeclareMathOperator{\chern}{ch}
\DeclareMathOperator{\td}{td}
\DeclareMathOperator{\Ob}{Ob}

\DeclareMathOperator{\Hom}{Hom}
\DeclareMathOperator{\Pic}{Pic}

\newcommand{\eg}{\emph{e.g.}~}
\newcommand{\ie}{\emph{i.e.}~}
\newcommand{\cf}{\emph{cf.}~}

\EpigaVolumeYear{6}{2022} \EpigaArticleNr{1}
\ReceivedOn{June 1, 2021}
\InFinalFormOn{September 28, 2021}
\AcceptedOn{October 25, 2021}

\title{Finite torsors over strongly $F$-regular singularities}
\titlemark{Finite torsors over strongly $F$-regular singularities}

\author{Javier A. Carvajal-Rojas}
\address{%
  \'Ecole Polytechnique F\'ed\'erale de Lausanne, SB MATH CAG, MA C3
  615 (B\^atiment MA), Station 8, CH-1015 Lausanne, Switzerland
  \newline
  Universidad de Costa Rica, Escuela de Matem\'atica, San Jos\'e
  11501, Costa Rica} 
\email{javier.carvajalrojas@epfl.ch}

\authormark{J.~Carvajal-Rojas}

\AbstractInEnglish{We investigate finite torsors over big opens of
spectra of strongly $F$-regular germs that do not extend to torsors
over the whole spectrum. Let $(R,\mmm,\kay,K)$ be a strongly $F$-regular
$\kay$-germ where $\kay$ is an algebraically closed field of
characteristic $p>0$. We prove the existence of a finite local cover
$R \subset R^{\star}$ so that $R^{\star}$ is a strongly $F$-regular
$\kay$-germ and: for all finite algebraic groups $G/\kay$ with
solvable neutral component, every $G$-torsor over a big open of 
$\Spec R^{\star}$ extends to a $G$-torsor everywhere.  To achieve
this, we obtain a generalized transformation rule for the
$F$-signature under finite local extensions. Such formula is used to
show that the torsion of $\Cl R$ is bounded by $1/s(R)$. By
taking cones, we conclude that the Picard group of globally
$F$-regular varieties is torsion-free. Likewise, this shows that
canonical covers of $\Q$-Gorenstein strongly $F$-regular singularities
are strongly $F$-regular.}

\MSCclass{14B05; 14L15; 13A35; 13A50; 13B05; 14F35; 14L30; 16T05}

\KeyWords{$F$-regularity; $F$-signature; finite torsors; local Nori
  fundamental group-scheme}

\acknowledgement{The author was supported in part by the NSF FRG Grant
  DMS \#1265261/1501115 and the ERC-STG \#804334.}

\begin{document}



\maketitle

\begin{prelims}

\DisplayAbstractInEnglish

\bigskip

\DisplayKeyWords

\medskip

\DisplayMSCclass

\end{prelims}


\newpage
\setcounter{tocdepth}{1}
\tableofcontents


\section{Introduction}

Inspired by the relationship between strongly $F$-regular
singularities in characteristic $p>0$ and KLT singularities in
characteristic zero, K.~Schwede, K.~Tucker, and the author obtained in
\cite[Theorem~A]{CarvajalSchwedeTuckerEtaleFundFsignature} an analog
of \cite[Theorem~1]{XuFinitenessOfFundGroups}. Concretely, they proved
that the order of the \'etale fundamental group of a big open of the
spectrum of a strongly $F$-regular germ $(R, \mmm, \kay, K)$ is at
most $1/s(R)$ and prime-to-$p$ (where $s(R)$ is the $F$-signature of
$R$). It is worth noting that B.~Bhatt, O.~Gabber, and M.~Olsson
introduced a spreading out technique to give a proof of Xu's result
from ours \cite{BhattGabberOlssonSpreadOut}. See
\autoref{sec.FsignAndSFR} for more on strong $F$-regularity and
$F$-signatures.

The aforementioned result gives a satisfactory answer to what extent
there are nontrivial finite \'etale torsors over big opens $U \subset
X \coloneqq \Spec R$ (with $R$ as above). We use the word torsor as in
\cite[III, \S4]{MilneEtaleCohomology} (\textit{i.e.} principal homogeneous
spaces in the fppf topology). Let us refer to a finite torsor over a
big open as a quasitorsor.  Indeed, \cite[Theorem
A]{CarvajalSchwedeTuckerEtaleFundFsignature} establishes the existence
of a generically Galois finite extension $(R,\mmm, \kay,K)\subset
(R^{\mathrm{\acute{e}t}}, \mmm^{\mathrm{\acute{e}t}},\kay,
K^{\mathrm{\acute{e}t}})$ that restricts to a
$\textnormal{Gal}(K^{\mathrm{\acute{e}t}}/K)$-torsor over
$X_{\mathrm{reg}}$ and: 
\begin{itemize}
\item $(R^{\mathrm{\acute{e}t}}, \mmm^{\mathrm{\acute{e}t}},
  \kay,K^{\mathrm{\acute{e}t}})$ is a strongly $F$-regular germ, 
\item every \'etale quasitorsor over $X^{\mathrm{\acute{e}t}}
  \coloneqq \Spec R^{\mathrm{\acute{e}t}}$ extends to a torsor over
  $X^{\mathrm{\Acute{e}t}}$, and 
\item 
its generic degree (\ie the order of
$\textnormal{Gal}(K^{\mathrm{\acute{e}t}}/K)$) is at most $1/s(R)$ and
prime-to-$p$. 
\end{itemize}

Cartier's theorem establishes that, in characteristic zero, all affine (resp. finite) algebraic groups are smooth (resp. \'etale); see \cite[3.23, 3.24]{MilneAlgebraicGroups}, \cite[II, \S6.1]{DemazureGrabielGroupes}. This result fails in positive characteristics, thereby showing that the local \'etale fundamental group is not suited to classify quasitorsors over a singularity (as it does in characteristic zero). This had been observed in the failure in positive characteristics of the celebrated Mumford--Flenner theorem characterizing the regularity of normal surface singularities by the triviality of their local fundamental group \cite{MumfordFundGroup,FlennerReineLokaleRinge}. Remarkably,  H.~Esnault and E.~Viehweg gave an explanation for this failure via their local Nori fundamental group-scheme. Their central idea was incorporating general quasitorsors into the picture; see \cite{EsnaultViehwegSurfaceSingularitiesDominatedSmoothVarieties}.  This raises the question of whether there is a \emph{maximal cover} as above for general finite torsors.

\begin{question} \label{que:MainQiestion}
Further assume $\kay \subset R$. Does there exist a \emph{nice} finite cover $R \subset R^{\star}$ by a strongly $F$-regular germ so that the restriction map of (isomorphism classes of) torsors
\[
\varrho^1_{X^{\star},U^{\star}}(G) \mathrel{:} \check{H}^1(X_{\mathrm{fl}}^{\star},G) \to \check{H}^1(U_{\mathrm{fl}}^{\star}, G)
\]
is surjective for all big opens $U^{\star} \subset X^{\star}=\Spec R^{\star}$ and all finite algebraic groups $G/\kay$?
\end{question}
By nice in \autoref{que:MainQiestion}, we seek for a cover described in terms of quasitorsors.  There is a purity theorem for finite torsors \cite[II, Proposition 7]{NoriFundamentalGroupScheme}, \cite{MoretBaillyPurete}; see \cite{MarramaPurityForTorsors} for an account.  That is,  \autoref{que:MainQiestion} has an affirmative answer if $R$ is regular by taking $R^{\star} = R$.  Thus,  the abundance of quasitorsors over a (normal) germ that are not restrictions of a torsor over the whole spectrum is a measurement of the severity of the singularity.  See \cite{EsnaultViehwegSurfaceSingularitiesDominatedSmoothVarieties}. 

Note that $\check{H}^1(X_{\mathrm{fl}},G) \neq 0$ unless $G/\kay$ is \'etale. See \cite[\S2]{EsnaultViehwegSurfaceSingularitiesDominatedSmoothVarieties} for a glimpse into the difficulties this introduces.  We may then rephrase \cite[Theorem A]{CarvajalSchwedeTuckerEtaleFundFsignature} by saying that \autoref{que:MainQiestion} has an affirmative answer for all \'etale algebraic groups $G/\kay$. In particular, we may replace $R$ by $R^{\mathrm{\acute{e}t}}$ and study \autoref{que:MainQiestion} for infinitesimal algebraic groups. In such a case, we answer \autoref{que:MainQiestion} affirmatively for all \emph{solvable} $G/\kay$ and $R \subset R^{\star}$ can be obtained as finite chain of purely inseparable cyclic covers. In the process, we will solve a fundamental problem raised by K.~Watanabe \cite{WatanabeFRegularFPure} on the perseverance of strong $F$-regularity under canonical covers and prove that the Picard group of a globally $F$-regular variety is torsion-free. To explain this in more detail, let us consider the following setup.

\begin{setup} \label{setup}
Let $(R, \mmm, \kay, K)$ be an $F$-finite strongly $F$-regular germ defined over an algebraically closed field $\kay$ of characteristic $p>0$. In particular, $R$ is an excellent henselian local normal domain with maximal ideal $\mmm$, residue field $\kay$, and field of fractions $K$.  We assume $d\coloneqq \dim R \geq 2$. Fix $Z = V(I) \subset \Spec R\eqqcolon X$ of codimension $\geq 2$ and complement $U\subset X$. Let $x$ denote the closed point of $X$.
\end{setup}

\begin{terminology} \label{term.QuasiTorsor}
Let $G/\kay$ be a finite algebraic group.  An extension of excellent $\kay$-algebras $A \subset B$ is a $G$-quasitorsor if $G$ acts on $B$ so that $A=B^G$ and the $G$-torsor locus of the quotient $q \mathrel{:}\Spec B \to \Spec A$ is a big open of $\Spec A$. By the purity theorem \cite{MoretBaillyPurete}, $G$-quasitorsors restrict to $G$-torsors over the regular locus if the base is further $\mathbf{R}_1$.
\end{terminology}

 Letting $G^{\circ} \subset G$ denote the neutral component, our main result is the following.

\begin{mainthm*}
[\autoref{thm.FundamentalGroupIsFinite}] Work in \autoref{setup}. There is a finite chain of finite local extensions of strongly $F$-regular $\kay$-germs
\[
(R, \mmm, \kay,K) \subsetneq (R_1, \mmm_1, \kay,K_1) \subsetneq \cdots \subsetneq (R_t, \mmm_t, \kay, K_t) = (R^{\star}, \mmm^{\star}, \kay, K^{\star}),
\]
where each $(R_i, \mmm_i, \kay,K_i) \subsetneq (R_{i+1}, \mmm_{i+1}, \kay, K_{i+1})$ is a $G_i$-quasitorsor and:
\begin{enumerate}
\item $G_i/\kay$ is a linearly reductive finite algebraic group, 
\item $[K^{\star}:K]$ is at most $1/s(R)$,
\item $\varrho^1_{X^{\star},U^{\star}}(G) \mathrel{:} \check{H}^1(X_{\mathrm{fl}}^{\star},G) \to \check{H}^1(U_{\mathrm{fl}}^{\star}, G)$ is surjective for all big opens $U^{\star} \subset X^{\star}=\Spec R^{\star}$ and all finite algebraic groups $G/\kay$ for which $G^{\circ}$ is solvable.
\end{enumerate}
\end{mainthm*}

In the rest of this introduction, we explain the main steps in achieving Theorem A, whereby outlining the paper and pinpointing some important results obtained in its afterglow. 

In studying \autoref{que:MainQiestion}, it is natural to stick as much as possible to the strategy that was so successful in \cite{CarvajalSchwedeTuckerEtaleFundFsignature} for the \'etale case. That strategy was based on bounding above the degree of a connected element of $\check{H}^1(U_{\mathrm{fl}},G)$ (for any \'etale $G/\kay$). This was done via a novel transformation rule for the $F$-signature under local \'etale quasitorsors; see \cite[Theorem B]{CarvajalSchwedeTuckerEtaleFundFsignature}. Thus,  one may hope to get an analogous transformation rule for general local quasitorsors.  Under closer inspection,  one sees that such transformation was built upon the \emph{trace map} of such local \'etale quasitorsors and exploiting its most basic properties. Thus, we are left facing two questions. Do such trace maps exist for general quasitorsors? If so, do they have the same or analogous properties to those of the trace map in the \'etale case? We dedicate \autoref{Trace} entirely to these questions.  Fortunately,  the former question will turn out to have a quite satisfactory answer whereas the latter will end up making things very interesting.  In \autoref{Trace},  we adapt to our purposes the theory of \emph{integrals and traces} found in the Hopf algebras literature.  We further prove the properties we need from these traces but could not find a reference for.  The following theorem summarizes the section.  We denote by $o(G)$ the order of a finite algebraic group $G/\kay$.

\begin{theoremA*}[{\autoref{ConstructionTrace},
    \autoref{thm.Torsor_and_trace}, \autoref{cor.Tr_generates}}]
Let $q\mathrel{:}Y \to X$ be a quotient by a finite algebraic group
$G/\kay$; see \autoref{quotients}. There exists $\Tr_{Y/X} \in
\Hom_X(q_*\sO_Y,\sO_X)$ such that: $q$ is $G$-torsor at $x \in X$ if
and only if $(q_* \sO_Y)_x$ is a free $\sO_{X,x}$-module of rank
$o(G)$ and $\Tr_{Y/X} \otimes \sO_{X,x}$ freely generates
$\Hom_{\sO_{X,x}}((q_*\sO_{Y})_x, \sO_{X,x})$ as a
$(q_*\sO_{Y})_x$-module. In particular, assuming that $X$, $Y$ are
$\mathbf{S}_2$ schemes and $q$ is a $G$-quasitorsor,  the
$\sO_Y$-module $q^!\sO_X = \sHom_{X}(q_* \sO_Y, \sO_X)$ is freely
generated by its global section $\Tr_{Y/X}$.  
\end{theoremA*}

Using this, we are able to obtain the following transformation rule.

\begin{theoremB*}
[{\autoref{cor.Tr_generates}, \autoref{thm.TransformationRule}}]
Let $(R, \mmm, \kay, K) \subset (S, \nnn, \el, L)$ be a local
$G$-quasitorsor between $\mathbf{S}_2$ domains. If $\Tr_{S/R}$ is
surjective and $\Tr_{S/R}(\nnn) \subset \mmm$,  then $\left[\el:
\kay\right]  \cdot s(S)= o(G) \cdot s(R)$. 
\end{theoremB*}

This formula can be used along the same lines of \cite{CarvajalSchwedeTuckerEtaleFundFsignature} to investigate \autoref{que:MainQiestion} modulo these two hypotheses on $\Tr_{S/R}$.  Just as in \cite[\S 2.3]{CarvajalSchwedeTuckerEtaleFundFsignature}, the surjectivity of the trace (\textit{i.e.} cohomological tameness \cite{KerzSchmidtOnDifferentNotionsOfTameness,ChinburgErezPappasTaylorTameActions}) can be granted by $R$ being a splinter using Theorem A and so this is not an issue for us either. However, the condition $\Tr_{S/R}(\nnn) \subset \mmm$ is rather troublesome; see \autoref{ex.ExampleBadTrace}.  Recall such condition is automatic in the \'etale case; see \cite[Lemma 2.10]{CarvajalSchwedeTuckerEtaleFundFsignature}. This work is largely concerned with understanding such property in the infinitesimal case.  Such analysis occupies most of \autoref{ExistenceofMaximalCover}. To see how we do it, consider the commutative infinitesimal case. There are two diametrically opposite classes of such groups: the linearly reductive groups (\textit{e.g.} $\mu_p$) and the unipotent ones (\textit{e.g.} $\alpha_p$). We can make $\Tr_{S/R}(\nnn) \subset \mmm$ work in the most critical class of linearly reductive quasitorsors. Namely, Veronese-type cyclic covers (yet those of Kummer-type may still be problematic); see \autoref{ter.VeroneseKummerTypes}. This suffices to solve a longstanding question raised by K.~Watanabe \cite{WatanabeFRegularFPure} on whether canonical covers of strongly $F$-regular singularities are strongly $F$-regular:

\begin{theoremC*}
[\autoref{pro.GoodTrace}, \autoref{ex.CanonicalCovers}]
Let $(R,\mmm,\kay,K)$ be a strongly $F$-regular (resp. $F$-pure) local
$\mathbb{F}_p$-domain.  Let $D$ be a $\mathbb{Q}$-Cartier divisor on
$\Spec R$ of index $n \in \mathbb{N}$ and write $\divisor a + nD =0$
for some $a \in K^{\times}$. Then, the corresponding cyclic cover
$R\subset C=\bigoplus_{i=0}^{n-1}R(iD)$ is strongly $F$-regular
(resp. $F$-pure) and $s(C)=n\cdot s(R)$. In particular, canonical
covers (also known as index-$1$ covers) of $\Q$-Gorenstein strongly
$F$-regular (resp. $F$-pure) singularities are strongly $F$-regular
(resp. $F$-pure) even if $p$ divides the Gorenstein index. 
\end{theoremC*}

Despite the success we could have in the linearly reductive case, the
unipotent case stands firmly as a potential truly troublesome case. We
deal with it in a rather radical way: 

\begin{theoremD*}[\autoref{thm.UnipotentTorsors}] Work in \autoref{setup}. Then, $\varrho_{X,U}^1(G) \mathrel{:} \check{H}^1(X_{\mathrm{fl}},G) \to \check{H}^1(U_{\mathrm{fl}},G)$ is surjective for all unipotent (not necessarily commutative) finite algebraic groups $G/\kay$.
\end{theoremD*}

This result should be thought of as a vast generalization of \cite[Corollary 2.11]{CarvajalSchwedeTuckerEtaleFundFsignature}, for the $p$-groups are the \'etale unipotent algebraic groups. Thus, it should be thought of as a generalization of the tameness of $\pi_1^{\textnormal{\'et}}(U)$ in \cite[Theorem A]{CarvajalSchwedeTuckerEtaleFundFsignature}. Putting together our observations from the linearly reductive as well as the unipotent cases, we get the following result which is the essential step in obtaining our Main Theorem inductively.

\begin{theoremE*}[\autoref{thm.MainTheorem}] Work in \autoref{setup}. Suppose that $\varrho_{X,U}^1(G)$ is not surjective for some finite algebraic group $G/\kay$ such that $G^{\circ}$ is commutative. There exist a nontrivial linearly reductive finite algebraic group $G'/\kay$ and $(R,\mmm,\kay,K) \subset (S,\nnn,\kay, L)$ a $G'$-torsor over $U$ but not everywhere such that $S$ is a strongly $F$-regular $\kay$-germ with $s(S) = o(G')\cdot s(R)$. 
\end{theoremE*}

Another instance of \'etale quasitorsors missing part of the picture in characteristic $p>0$ is that we were able to bound only the prime-to-$p$ torsion of $\Cl R$ in \cite{CarvajalSchwedeTuckerEtaleFundFsignature}. Here, we complete the picture by bounding all torsion. By taking cones, we get that globally $F$-regular varieties have torsion-free divisor class group. We discuss these applications in \autoref{sec:applications}.  

\begin{theoremF*}
[\autoref{cor.TorsionPicard}, \autoref{cor.DivisorClassGroup}] Work in \autoref{setup}.
The torsion of $\Cl R$ is bounded by $1/s(R)$. In particular, $\Cl R
$ is torsion-free if $s(R)>1/2$. Moreover, if $Y$ is a globally
$F$-regular $\kay$-variety then $\Pic Y$ is torsion-free.
\end{theoremF*}

\subsection*{Acknowledgements}
The author would like to thank Marco~Antei, Fabio Bernasconi,
Bhargav~Bhatt, Srikanth~Iyengar, Christian~Liedtke, Zsolt~Patakfalvi,
Stefan~Patrikis, Anurag~Singh, Daniel~Smolkin, Axel~St\"abler, and
Kevin~Tucker for all the help and valuable discussions they provided
to him when working on this paper. I also would like to thank Manuel
Blickle, Holger Brenner, Patrick Graf, Mircea Musta\c{t}\u{a}, and
Thomas Polstra for valuable comments on previous drafts. I am also
thankful to Abhishek Shukla who pointed out a mistake on the first
version.  I am deeply grateful to my advisor Karl~Schwede for all his
support, guidance, and encouragement throughout this project.  Without
his helpful advice and insights this preprint would have not being
possible. I am particularly thankful to the anonymous referees for
their valuable feedback.

\section{Notations, conventions, and preliminaries}
\begin{convention}
All schemes are defined over an algebraically closed field $\kay$ of
characteristic $p>0$. We let $F \mathrel{:} X \to X$ denote the Frobenius endomorphism of
a scheme $X$. For instance, if $R$ is a ring then $F^e_* R$ denotes
the $R$-module obtained by restriction of scalars under $R \to R$, $r
\mapsto r^{p^e}$. Let us recall that a scheme $X$ is said to be $F$-finite if its Frobenius endomorphism is finite. We assume all our schemes to be $F$-finite and noetherian, and so
excellent. Fibered/tensor products are defined over $\kay$
unless otherwise explicitly stated.
\end{convention}

\subsection{$F$-signatures and strong
  $F$-regularity} \label{sec.FsignAndSFR} In view of Kunz's theorem
\cite{KunzCharacterizationsOfRegularLocalRings}, it is natural to use
Frobenius splittings to measure singularities in positive
characteristic. Strongly $F$-regular singularities were introduced by
M.~Hochster and C.~Huneke
\cite{HochsterHunekeTightClosureAndStrongFRegularity}. Let
$(R,\mmm,\kay,K)$ be a local ring. The \emph{$e^\mathrm{th}$ $F$-splitting
  number of $R$} is defined by 
\[
  a_e(R) \coloneqq
  \lambda_R \big(\Hom_R(F^e_* R, R)\big/\Hom_R(F^e_* R, \mmm)\big)
  \in \N
\]
where $\lambda_R(M)$ denotes the length of an $R$-module $M$. This
number turns out to be the maximum of the set
$\{a \mid\/ \exists\, F^e_*R \to R^{\oplus a}
\textnormal{ surjective}\} \subset \Z$;
see \cite{BlickleSchwedeTuckerFSigPairs1}. In particular, setting
$\delta \coloneqq \dim R + \log_p [\kay^{1/p}:\kay]$, we have $0 \leq
a_e(R) \leq p^{e \delta}$ where $p^{e \delta}$ is the (generic) rank
of $F^e_*R$ and at the same time the free rank of $F^e_*R$ if $R$ were
regular. For details, see
\cite{BlickleSchwedeTuckerFSigPairs1,TuckerPolstraUniformApproach}.

In this way, we may reinterpret Kunz's theorem by saying that: $R$ is regular if and only if $a_e(R) = p^{e\delta}$ for some (equivalently all) $e \geq 1$. Moreover, one says that $R$ is \emph{$F$-pure} if $a_e \neq 0$ for some (equivalently all) $e \geq 0$. To define \emph{strong $F$-regularity}, one considers the asymptotic behavior of the splitting numbers as $e \rightarrow \infty$. To be precise, after \cite{HunekeLeuschkeTwoTheoremsAboutMaximal,SmithVanDenBerghSimplicityOfDiff}, one defines the $F$-signature of $R$ as the limit
\[
s(R) \coloneqq \lim_{e \rightarrow \infty} \frac{a_e(R)}{p^{e\delta}} \in [0,1],
\]
whose existence was established in \cite{TuckerFSigExists}. Then, $R$ is \emph{strongly $F$-regular} if and only if $s(R) > 0$; see \cite[Theorem 0.2]{AberbachLeuschke}. On the other hand, it turns out that $s(R)=1$ if and only if $R$ is regular; see \cite[Corollary~16]{HunekeLeuschkeTwoTheoremsAboutMaximal}.

Intuitively, one often thinks of $s(R)$ as a volume attached to the singularity $R$, \cf \cite{VonKorffFSigOfAffineToric}. Furthermore, $F$-signatures have been linked (conjecturally) to normalized volumes; see \cite[\S6.3.1]{ChiLiuChenyangNormalizedVolumes}.  For surveys on the subject of strong $F$-regularity, we recommend \cite{SmithZhang,SchwedeTuckerTestIdealSurvey}. See \cite{TuckerPolstraUniformApproach} for more on $F$-signatures.

\subsection{Linear and bilinear forms} \label{Forms} Let $M$ be a module over a commutative ring $R$. We denote its dual $\Hom_R(M,R)$ by $M^{\vee}$. A \emph{bilinear form} on $M$ over $R$ is the same as an element of $(M\otimes_R M)^{\vee}\eqqcolon  M^{\vee 2}$. By $\otimes$-$\Hom$ adjointness, there are two natural isomorphisms $\upsilon_i\mathrel{:}  M^{\vee 2} \to \Hom_R(M, M^{\vee})$. A bilinear form $\Theta$ is \emph{symmetric} if $\upsilon_1 (\Theta ) = \upsilon_2 (\Theta)$, and we write $\upsilon (\Theta)$ for either one. It is said to be \emph{nondegenerate} (resp. \emph{nonsingular}) if $\upsilon (\Theta )$ is injective (resp. an isomorphism). If $M$ is free of finite rank, there is a determinant function $\det\mathrel{:} M^{\vee 2} \to R$. Note that $\Theta$ is nondegenerate (resp. nonsingular) if and only if $\det \Theta$ is a nonzerodivisor (resp. a unit). If $M$ is locally free of finite rank, we associate to $\Theta$ a locally principal ideal $\det \Theta $ of $R$, for naturally $(M^{\vee 2})_{\mathfrak{p}} \cong (M_{\mathfrak{p}})^{\vee 2} $. Say $S=M$ is an $R$-algebra, \ie there is a diagonal morphism $\Delta_{S/R}\mathrel{:}S \otimes_R S \to S$. Its dual $\Delta_{S/R}^{\vee}\mathrel{:} S^{\vee} \to (S \otimes_R S)^{\vee}$ gives a canonical way to define a bilinear form from a linear one. We refer to $\theta \in S^{\vee}$ as nondegenerate or nonsingular if so is $\Delta_{S/R}^{\vee}(\theta)$. If $S/R$ is finite and flat, one defines the \emph{discriminant} of $\theta$ as $\disc \theta \coloneqq \det \Delta_{S/R}^{\vee}(\theta)$.

\subsection{Group-schemes and their actions on schemes}
For the most part, we follow
\cite{MilneAlgebraicGroups,TateFiniteFlatGroupSchemes,MontgomeryHopfAlgebras}. Let
$G/\kay$ be an affine group-scheme. We denote by $\sO(G)$ its
corresponding (commutative) Hopf algebra. Given a Hopf algebra
$H/\kay$, we denote by $u\mathrel{:} \kay \to H $ and
$\Delta\mathrel{:} H \otimes H \to H $ the algebra morphisms of
\emph{unit} and \emph{product} whereas the coalgebra morphisms
\emph{counit} and $\emph{coproduct}$ are denoted by $e\mathrel{:} H
\to \kay $ and $\nabla\mathrel{:} H \to H \otimes H$. For the
\emph{antipode}, we use $\iota \mathrel{:} H \to H$. The trivial
group-scheme is denoted by $\ast$. If $G/\kay$ is finite, its
\emph{order} is $o(G)\coloneqq\dim_{\kay} \sO(G)$. If $G$ is
commutative,  its \emph{Cartier dual} is denoted by $G^{\vee}$; see
\cite[11.c]{MilneAlgebraicGroups}. The dual of a finite Hopf algebra
$H$ is denoted by $H^{\vee}$; see
\cite[\S1.2]{MontgomeryHopfAlgebras}. In particular
$\sO(G^{\vee})=\sO(G)^{\vee}$. The connected-\'etale exact sequence
\cite[5.h]{MilneAlgebraicGroups} is denoted by 
\[
\ast \to G^{\circ} \to G \to \pi_0(G) \to \ast,
\]
which splits as $\kay$ is perfect \cite[11.3]{MilneAlgebraicGroups}. In other words, $G=G^{\circ} \rtimes \pi_0(G)$. In fact, $G_{\mathrm{red}} \cong \pi_0(G)$ and the splitting is (uniquely) given by the canonical closed embedding $G_{\mathrm{red}} \to G$.
 
\subsubsection{Examples of affine algebraic groups} 
Finite \emph{constant group-schemes} are examples of finite \'etale algebraic groups and (since $\kay$ is separably closed) these are the only ones \cite[2.b]{MilneAlgebraicGroups}. If $G$ is a finite discrete group, the Hopf algebra of the constant group-scheme it defines (also denoted by $G$) is $\sO(G)=\Hom_{\mathsf{Set}}(G,\kay)$. The coproduct is defined by 
\begin{align*}
\nabla\mathrel{:}\Hom_{\mathsf{Set}}(G,\kay) &\to  \Hom_{\mathsf{Set}}(G \times G,\kay) \xleftarrow{\cong} \Hom_{\mathsf{Set}}(G,\kay) \otimes \Hom_{\mathsf{Set}}(G,\kay),\\
\gamma &\mapsto ((g,h)\mapsto \gamma(gh) )
\end{align*} 
The counit is the evaluation-at-$1$ map and the antipode is given by: $\iota(\gamma)(g)\coloneqq\gamma(g^{-1})$ for all $\gamma \in \sO(G)$, $g\in G$ \cite[2.3, Exercise 3-1]{MilneAlgebraicGroups}. 

From a group $\Gamma$, one constructs another Hopf algebra as follows. The underlying algebra is $\kay[\Gamma]$---the group $\kay$-algebra of $\Gamma$. The coproduct, counit, and antipode are given by $\nabla(g)=g \otimes g$, $e(g)=1$, $ \iota(g)=g^{-1}$,  for all $g\in \Gamma$.  If $\Gamma$ is abelian and finitely generated, we denote the corresponding algebraic group by $D(\Gamma)$ and refer to them as the \emph{diagonalizable} algebraic groups \cite[Chapter~12]{MilneAlgebraicGroups}. For example, $D(\Z)$ is the \emph{multiplicative group} $\mathbb{G}_{\textnormal{m}}$, with $\sO(\mathbb{G}_{\textnormal{m}})=\kay[\zeta, \zeta^{-1}]$. Also, $D(\Z/n)$ is the \emph{group of $n$-th roots of unity} $\mu_n$, with $\sO(\mu_n) = \kay[\zeta]/(\zeta^n-1)$.  Likewise,  $D(\Z^{\oplus n})=\mathbb{D}_n$ is the group of invertible diagonal $n\times n$ matrices. In fact, $D$ is an exact contravariant functor \cite[12.9]{MilneAlgebraicGroups}. Applying it to $
0 \to \Z \xrightarrow{\cdot n} \Z \to \Z/n\to 0 $ yields the \emph{Kummer sequence}:
\begin{equation} \label{eqn.KummerSES}
\ast \to \mu_n \to \mathbb{G}_{\textnormal{m}} \to \mathbb{G}_{\textnormal{m}} \to \ast.
\end{equation}

The \emph{additive group} $\mathbb{G}_{\textnormal{a}}$ (see \cite[2.1]{MilneAlgebraicGroups}) is represented by the Hopf algebra $\kay[\xi]$ where $\nabla(\xi)=\xi \otimes 1 + 1 \otimes \xi$, $e(\xi)=0$ and $\iota(\xi)=-\xi$. Its $e^\mathrm{th}$ $\kay$-linear Frobenius homomorphism $\sO(\mathbb{G}_{\textnormal{a}}) \to \sO(\mathbb{G}_{\textnormal{a}})$ is given by $\xi \mapsto \xi^{p^e}$. It is an injective homomorphism of Hopf algebras and so it is faithfully flat \cite[3.i]{MilneAlgebraicGroups}. The kernel of the corresponding quotient homomorphism $F^e\mathrel{:}\mathbb{G}_{\textnormal{a}} \to \mathbb{G}_{\textnormal{a}}$ is the Cartier self-dual \emph{infinitesimal group-scheme} $\alpha_{p^{e}}$ \cite[2.5]{MilneAlgebraicGroups}, which is represented by $\sO(\alpha_{p^e})=\kay[\xi]/\xi^{p^e}$. Thus, we have the following short exact sequence
\begin{equation} \label{eqn.InfinitesimalSES}
\ast \to \alpha_{p^e} \to \mathbb{G}_{\textnormal{a}}
\xrightarrow{F^e} \mathbb{G}_{\textnormal{a}} \to \ast.
\end{equation}
Likewise, we have the \emph{Artin--Schreier exact sequence}
\begin{equation} \label{eqn.ArtinSchreierSES}
\ast \to \Z/p \to \mathbb{G}_{\textnormal{a}} \xrightarrow{F-\id}
\mathbb{G}_{\textnormal{a}} \to \ast,
\end{equation}
where $F-\id$ is given by $\xi \mapsto \xi^p -\xi$ \cite[Exercise
14-3]{MilneAlgebraicGroups}. 

Let $\mathbb{T}_n \subset \textnormal{GL}_n$ be the subgroup of upper triangular matrices and $\mathbb{U}_n \subset \mathbb{T}_n$ be the one of upper triangular matrices with $1$ along the diagonal \cite[2.9]{MilneAlgebraicGroups}.  In fact, $\mathbb{U}_n$ is a normal subgroup of $\mathbb{T}_n $ whose quotient is $\mathbb{D}_n$ and $\mathbb{T}_n = \mathbb{U}_n \rtimes \mathbb{D}_n$.

\subsubsection{Trigonalizable groups} \label{sec.TrigonalizableGroups}
The \emph{coradical} $C_0$ of a coalgebra $C$ is defined as the sum of its simple subcoalgebras \cite[5.1.5]{MontgomeryHopfAlgebras}. This definition applies to a Hopf algebra as well under the caveat that its coradical may not be a Hopf subalgebra. However, we have:
\[
\kay \cdot 1 = \kay[\{1\}] \subset \kay[X(H)] \subset H_0 \subset H,
\]
where $X(H)=\{h \in H\mid \nabla(h)=h\otimes h\}$ is the abelian group of \emph{group-like elements} of $H$ (also known as the \emph{characters} of $G$ if $H=\sO(G)$ \cite[12.a]{MilneAlgebraicGroups}). See the discussion proceeding \cite[5.1.5]{MontgomeryHopfAlgebras}. An affine group-scheme $G/\kay$ is said to be \emph{unipotent} if $\kay \cdot 1 = \sO(G)_0$; see \cite[14.5]{MilneAlgebraicGroups}, \cite[\S 5.2]{MontgomeryHopfAlgebras}. On the opposite extreme, it is said to be \emph{linearly reductive} if $\sO(G)_0= \sO(G)$.\footnote{In other words,  $\sO(G)$ is \emph{cosemisimple} \cite[2.4.1]{MontgomeryHopfAlgebras}, which agrees with  \cite[12.l]{MilneAlgebraicGroups}; see \cite[2.4.6]{MontgomeryHopfAlgebras}.} If $\kay[X(G)]=\sO(G)_0$, $G$ is said to be \emph{trigonalizable}.\footnote{That is,  $G$ is trigonalizable if $\sO(G)$ is \emph{pointed}; see \cite[5.1.5]{MontgomeryHopfAlgebras} and the discussion proceeding it.  This coincides with \cite[16.1]{MilneAlgebraicGroups}; see \cite[Exercise 16-2]{MilneAlgebraicGroups}}  Diagonalizable group-schemes are those for which $\kay[X(G)]= \sO(G)$ \cite[12.7, 12.8]{MilneAlgebraicGroups},  and so they  are linearly reductive.  Both unipotent and diagonalizable groups are trigonalizable.

Since $\kay$ is separably closed,  we may refer to diagonalizable group-schemes as group-schemes of \emph{multiplicative type}; see \cite[12.8]{MilneAlgebraicGroups}.  These are the commutative linearly reductive groups \cite[12.54]{MilneAlgebraicGroups}.  See  \cite[12.18]{MilneAlgebraicGroups} for more. Nagata's theorem \cite{NagataCompleteReducibility} establishes that $G$ is linearly reductive if and only if $p \nmid o(\pi_0(G)) $ and $G^{\circ}$ is of multiplicative type; see \cite[Theorem 5.7.4]{MontgomeryHopfAlgebras}, \cite[12.56]{MilneAlgebraicGroups}. Thus,   $G$ is linearly reductive if and only if $\pi_0(G)$ has prime-to-$p$ order and $G^{\circ}=D(\Gamma)$ for some abelian group $\Gamma$ whose torsion is $p$-torsion.  

Unipotent groups are,  up to isomorphism,  the subgroups of $\mathbb{U}_n$ \cite[14.4]{MilneAlgebraicGroups}.  Now,  $\mathbb{U}_n$ has a central normal series (\cite[6.f]{MilneAlgebraicGroups}) whose intermediate quotients are canonically isomorphic to $\mathbb{G}_{\textnormal{a}}$.  Hence,  every unipotent group admits a central normal series whose quotients are (isomorphic to) subgroups of $\mathbb{G}_{\textnormal{a}}$ \cite[14.21]{MilneAlgebraicGroups}.  Thus, $\mathbb{G}_{\textnormal{a}}$ and its subgroups (\textit{e.g.}  $\alpha_p$ and $\Z/p$) are the building blocks for unipotent groups. Likewise,  the trigonalizable groups are the subgroups of $\mathbb{T}_n$ \cite[16.2]{MilneAlgebraicGroups}.  A trigonalizable group $G$ admits a factorization $G_{\mathrm{u}} \rtimes G_{\mathrm{mt}}$ where $G_{\mathrm{u}}$ is unipotent and $G_{\mathrm{mt}}$ is of multiplicative type; see \cite[16.6, 16.26]{MilneAlgebraicGroups}.

\subsubsection{Actions, quotients, and torsors}
 \label{def.torsors}
Let $Y/\kay$ be a scheme and $G/\kay$ be an affine group-scheme.  A \emph{(right) action} of $G$ on $Y$ is a $\kay$-morphism $a\mathrel{:}  Y \times G \to Y$ such that 
\[
\xymatrixcolsep{3pc}\xymatrix{
Y \times G \times G  \ar[r]^-{\id \times \nabla } \ar[d]_-{ a \times \id} & Y \times G \ar[d]^-{a}\\ 
Y \times G \ar[r]^-{a} & Y
}\qquad
\xymatrix{
Y \times \ast  \ar[r]^-{\id \times e} \ar[d]_-{\cong} & Y \times G \ar[dl]^-{a}\\
Y
}
\]
are commutative.  A $\kay$-morphism $f\mathrel{:}Y \to X$ is said to be \emph{$G$-invariant} if $a\mathrel{:} Y \times G \to Y$ is an $X$-morphism.  A \emph{(categorical) quotient} of $a\mathrel{:} Y \times G \to Y$ is a $G$-invariant morphism $q\mathrel{:} Y \to X$ that factors uniquely any other one.  If it exists,  it is unique up to unique isomorphism. See \cite[III, \S12]{MumfordAbelianVarieties}, \cite[IV]{AbelianVarietiesGeerMooen}, \cite[\S2.6]{BrionSomeStructureTheoremsForAlgebraicGroups}.

A \emph{$G$-torsor} is a principal $G$-bundle in the fppf-topology. That is,  it is an fppf morphism $q\mathrel{:}Y \to X$ together with an action $a\mathrel{:} Y \times G \to Y $ such that $q$ is $G$-invariant and the induced morphism $a \times \mathrm{pr}_1\mathrel{:} Y \times G \to Y \times_X Y$ is an isomorphism; see \cite[III, Proposition 4.1]{MilneEtaleCohomology}. Torsors are quotients as fppf maps are strict epimorphisms \cite[I, Theorem 2.17]{MilneEtaleCohomology}.  Since $G/\kay$ is affine,  the isomorphism classes of $G$-torsors over $X$ are naturally classified by the pointed-set $\check{H}^1(X_{\textnormal{fl}},G)$; see \cite[III, \S4]{MilneEtaleCohomology}, \cite{GiraudNonabelianCohomology}. The distinguished point is the class of the trivial $G$-torsor $\mathrm{pr}_1\mathrel{:} X \times G \to X$. If $\phi \mathrel{:} G \to H$ is a homomorphism,  one defines the map of pointed-sets $\check{H}^1(\phi) \mathrel{:} \check{H}^1(X_{\textnormal{fl}},G) \to \check{H}^1(X_{\textnormal{fl}},H)$ as follows \cite[III, 2.4.2.1]{GiraudNonabelianCohomology}. If $q\mathrel{:}Y \to X$ is a $G$-torsor, $\check{H}^1(\phi)(q)$ is the \emph{contracted product} $Y \wedge^{G} H $ \cite[III, \S1.3]{GiraudNonabelianCohomology}.  By definition, $Y \wedge^{G} H$ is the quotient of $Y \times H$ by the diagonal right action of $G$: $(y,h)\cdot g \mapsto \bigr(y\cdot g, \phi(g)^{-1} h\bigl)$:
\[
  Y \wedge^{G} H
  \coloneqq
  Y \times H \Big/ \raisebox{-.5ex}{$(yg, h) \sim (y,\phi(g)h)$}.
\]
The right action of $H$ on $Y \wedge^G H $ is given by the rule $(y,
h_1) \cdot h_2 = (y, h_2h_1)$. The morphism $Y \wedge^G H
\to X$, say $(y,h) \mapsto f(y)$, is an $H$-torsor under this action
\cite[III, 1.4.6, \cf 1.3.6]{GiraudNonabelianCohomology}.\footnote{In
  fact, $Y \wedge^G H \to X$ is the morphism corresponding to 
  $f\mathrel{:} Y \to X$ in the adjointness
  \cite[1.3.6.(iii)]{GiraudNonabelianCohomology}.} If $G$ is
commutative, $\check{H}^1(X_{\textnormal{fl}},G)$ coincides with the
derived-functor cohomology abelian group $H^1(X_{\textnormal{fl}},G)$
\cite[III, 3.5.4]{GiraudNonabelianCohomology}. Moreover, given a short
exact sequence
\[
\ast \to G' \to G \to G'' \to \ast
\]
we have an exact sequence of pointed sets
\[
* \to G'(X) \to G(X) \to G''(X) \xrightarrow{\delta^0}
\check{H}^1(X_{\textnormal{fl}},G') \to
\check{H}^1(X_{\textnormal{fl}},G) \to
\check{H}^1(X_{\textnormal{fl}},G'')
\]
which can be continued using Giraud's second nonabelian cohomology
\cite{GiraudNonabelianCohomology,DebremaekerNonAbelianCohomology}
\[
\cdots \to \check{H}^1(X_{\textnormal{fl}},G'')
\xrightarrow{\delta^1} \check{H}^2(X_{\textnormal{fl}},G')
\to \check{H}^2(X_{\textnormal{fl}},G) \to
\check{H}^2(X_{\textnormal{fl}},G'') 
\]
In the commutative case,  this agrees with the long exact sequence from derived-functor abelian cohomology with respect to the fppf site \cite[III, \S3.4]{GiraudNonabelianCohomology}. Thus, when $G$ is commutative, we shall often write $H^i(X_{\textnormal{fl}},G)=\check{H}^i(X_{\textnormal{fl}},G)$ for $i=0,1,2$.

\begin{remark} \label{rem:GaloisCorrespondence}
Let $G \rightarrow G''$ be a faithfully flat homomorphism of finite group-schemes over $\kay$ with kernel $G' \to G$. Let $Y \to X$ be a $G$-torsor. Then, the restricted action $Y \times G' \to Y$ realizes the canonical morphism $Y \to Y\wedge^G G'' \eqqcolon Y' $,  say $y \mapsto (y, e'')$,  as a $G'$-torsor. Indeed,  $Y\to Y'$ is a quotient of $Y$ by $G'$ \cite[3.2.5]{GiraudNonabelianCohomology}.  Since $G$ acts freely on $Y$ (\ie $Y \times G \to Y \times Y$ is a closed immersion) then so does $G'$.  Hence, $Y \times G' \to Y \times_{Y'} Y$ is an isomorphism \cite[3.1.2.(i)]{GiraudNonabelianCohomology}, and $Y \to Y'$ is fppf \cite[III, Theorem 1]{MumfordAbelianVarieties}, \cite[Theorem 4.16.(iii)]{AbelianVarietiesGeerMooen}.
\end{remark}

\begin{remark} \label{quotients}
Let $G/\kay$ be a finite algebraic group acting on a scheme
$Y/\kay$. If the orbit of every closed point of $Y$ is contained in an
affine chart (see \cite[IV]{AbelianVarietiesGeerMooen}) then the
quotient $q: Y \to X$ exists \cite[4.16, \cf
4.13]{AbelianVarietiesGeerMooen}. Also see the more classical reference \cite[III,\S2, Corollaire 6.1]{DemazureGrabielGroupes}. Moreover, $q \mathrel{:} Y \to X$ is
quasi-finite and integral (\cf \cite[\S4.2]{MontgomeryHopfAlgebras}),
and $q^{\#}\mathrel{:}\sO_X \to q_* \sO_Y$ induces an
isomorphism $\sO_X \xrightarrow{} (q_*\sO_Y)^G$.\footnote{%
  The sheaf $(q_*\sO_Y)^G$ is the subsheaf of $q_* \sO_Y$ of
  $G$-invariant functions; see \cite[4.12]{AbelianVarietiesGeerMooen},
  \cite[III, 12]{MumfordAbelianVarieties}.}
Thus, $q\mathrel{:} Y \to X$ is locally given by spectra of rings of
invariants.  \emph{By a finite $G$-quotient $q \mathrel{:} Y \to X$, we always
mean a morphism as above.} If the action of $G$ on $Y$ is free
(\textit{e.g.} $q$ is a $G$-torsor) then $q \mathrel{:} Y \to X$ is finite and flat.
Likewise, if $Y$, $X$ are $\mathbf{S}_2$ and $G$ acts freely on a big
open of $Y$ (\textit{e.g.}  $q$ is a quasitorsor), then $q$ is finite.
\end{remark}

Let us restate the axioms for actions in a way that will be useful later. We may base change $(G, \nabla, e, \iota)$ by $Y/\kay$ to get $(G_Y, \nabla_Y, e_Y, \iota_Y)$---a group-scheme over $Y$.  The two axioms for right actions translate into the commutative diagrams
\[
\xymatrix{
G_Y \times_Y  G_Y  \ar[r]^-{ \nabla_Y } \ar[d]_-{ a \times \id} & G_Y \ar[d]^-{a}\\ 
 G_Y \ar[r]^-{a} & Y
}\qquad
\xymatrix{
Y \ar[r]^-{e_Y} \ar[d]_-{\id} & G_Y \ar[dl]^-{a}\\
Y
}
\]
If $Y= \Spec S$, a (right) action of $G$ on $Y$ is a (right) \emph{coaction} of $\sO(G)$ on $S$, \ie a homomorphism $\alpha \coloneqq a^{\#}\mathrel{:} S \to S \otimes \sO(G)$ satisfying the following two commutative diagrams:
\[
\xymatrixcolsep{3pc}\xymatrix{
S \otimes \sO(G) \ar[r]^-{\id \otimes \nabla} & S\otimes \sO(G) \otimes \sO(G) \\
S \ar[r]^-{\alpha} \ar[u]^-{\alpha} & S \otimes \sO(G) \ar[u]_-{\alpha \otimes \id}
} \qquad
\xymatrix{
S \otimes \kay & S \otimes \sO(G) \ar[l]_-{\id \otimes e}\\
S \ar[u]^-{\cong} \ar[ru]_{\alpha} &
}
\]
The \emph{ring of coinvariants} is $S^{G} \coloneqq \{s\in S \mid \alpha(s)=s\otimes 1\}$. As before, base changing $\sO(G)$ by $S/\kay$ gives the Hopf $S$-algebra $\sO(G_S)$ representing the group-scheme $G_S/S$. The coproduct $\nabla_S$ is given by composition of $\id \otimes \nabla $ with $S \otimes \sO(G) \otimes \sO(G) \xrightarrow{\cong} (S\otimes \sO(G)) \otimes_S (S\otimes \sO(G)) $ and the identity $e_S$ is given by $\id  \otimes e$. Thus, the coaction axioms are the commutative diagrams:
\[
\xymatrixcolsep{3pc}\xymatrix{
\sO(G_S) \ar[r]^-{\nabla_S} & \sO(G_S) \otimes_S \sO(G_S)\\
S \ar[r]^-{\alpha} \ar[u]^-{\alpha} & \sO(G_S) \ar[u]_-{\alpha \otimes \id}
} \qquad
\xymatrix{
 S & \sO(G_ S) \ar[l]_-{e_S}\\
S \ar[u]^{\id} \ar[ru]_{\alpha} &
}
\]
In particular, $\alpha$ is injective by the second axiom.

\begin{remark} \label{rem.ActionsCoactions}
We said above that a (right) action of $G$ on $\Spec S$ is the same as a (right) coaction of $\sO(G)$ on $S$. This is further equivalent to a (left) action of $\sO(G)^{\vee}$ on $S$.  See \cite[Ch.  4]{MontgomeryHopfAlgebras}.  Roughly speaking, a (left) \emph{action} of a Hopf algebra $H$ on $S$ is a $\kay$-linear map $\beta\mathrel{:}H \otimes S \to S$ (\ie a left $H$-module structure on $S$) satisfying a pair of axioms dual to the ones we had for coactions. The \emph{ring of invariants} is defined by $\{s \in S \mid h \cdot s = e(h)s \text{ for all } h \in H\}$. If $H/\kay$ is finite,  an action (resp. coaction) of $H$ is the same as a coaction (resp. action) of $H^{\vee}$ in such a way that rings of invariants and coinvariants coincide. Indeed, if $H$ coacts on $S$ by $\alpha\mathrel{:}S \to S \otimes H$, then $H^{\vee}$ acts on $S$ via $\eta \cdot s = (\id \otimes \eta)(\alpha(s))$ for all $\eta \in H^{\vee}$, $s\in S$. It is clear that coinvariant elements are invariant. The converse is true yet it relies on $H/\kay$ being finite: to check two elements (\textit{e.g.} $\alpha(s)$ and $s\otimes 1$) in the \emph{finite rank free} $S$-module $S\otimes H$ are the same, it suffices to show that their images under $\id \otimes \eta$ agree for all $\eta \in H^{\vee}$ as these maps generate the $S$-dual module of $S \otimes H$.  Thus, it is \emph{a priori} easier to show that an element is an invariant than a coinvariant although these are \emph{a posteriori} equivalent due to finiteness.  Finally,  given an action $H^{\vee} \otimes S \to S$, its associated coaction is only defined after choosing a $\kay$-basis for $H$. If $h_1,\ldots ,h_d$ is a basis with corresponding dual basis $\eta_1,\ldots ,\eta_d$, then the coaction is given by the rule $s \mapsto \sum_{i=1}^d(\eta_i \cdot s) \otimes h_i$. 
\end{remark}

\section{Traces of quotients by finite group-schemes} \label{Trace}
Let $q\mathrel{:} X \to Y$ be a finite $G$-quotient as in \autoref{quotients}. In this section, we construct an $\sO_X$-linear map $\Tr_{Y/X}\mathrel{:} q_* \sO_Y \to \sO_X$ generalizing the classic trace map for \'etale $G/\kay$. Further, we show its properties necessary to study \autoref{que:MainQiestion} in the same vein of \cite{CarvajalSchwedeTuckerEtaleFundFsignature}. 

\subsection{Construction of the trace} \label{ConstructionTrace} We start with the affine case and obtain the general one by gluing on affine charts.  If such trace maps $\Tr_{S/S^G} \mathrel{:} S \to S^G$ exist, they must exist for the action of $G$ on itself, where $\sO(G)^G=\kay$. We discuss this core case first and construct $\Tr_{S/S^G}$ from $\Tr_{G/\kay}$ and the given action. Our first goal then is to explain the existence of a special $\kay$-linear map $\Tr_{G}=\Tr_{G/\kay}\mathrel{:}\sO(G) \to \kay$. Since $\sO(G)$ is a Gorenstein finite $\kay$-algebra \cite[11.f]{MilneAlgebraicGroups}, the $\sO(G)$-module $\omega_G=u^!\omega_{\kay}=u^! \kay=\Hom_\kay(\sO(G), \kay)$ is free of rank $1$. $\Tr_{G}$ is going to be a special generator of $\omega_G$. Note that $\sO(G)$ coacts on itself via the coproduct $\nabla \mathrel{:} \sO(G) \to \sO(G) \otimes \sO(G)$, which means that $\sO(G)^{\vee}$ acts on $\sO(G)$. In fact, $g\cdot \gamma  = (\id \otimes g)\bigr(\nabla(\gamma)\bigl) $ for all $g\in \sO(G)^{\vee}$, $\gamma \in \sO(G)$. We want $\Tr_G \in \sO(G)^{\vee}$ to yield invariants when acting on elements via this action. That is, we need $\Tr_G \cdot \gamma$ to be an invariant for all $\gamma \in \sO(G)$, \textit{i.e.}:
$g\cdot(\Tr_G \cdot \gamma ) = g(1) (\Tr_G \cdot \gamma)$, for all  $g\in \sO(G)^{\vee}$. This gives the desired property: 
\[
g \cdot \Tr_G = e^{\vee}(g) \cdot \Tr_G, \qquad \forall g\in \sO(G)^{\vee}.
\] 
Following Hopf algebras nomenclature, we are requiring $\Tr_G$ to be a \emph{left integral} of the Hopf algebra $\sO(G)^{\vee}$. To the best of the author's knowledge, this concept was introduced by R.~Larson and M.~Sweedler in \cite{LarsonSweedlerTrace}. We summarize their definition, existence, and uniqueness in the following theorem. See  \cite[Ch. 2]{MontgomeryHopfAlgebras} for further details.

\begin{theorem}[{\cite[2.1.1,  2.1.3]{MontgomeryHopfAlgebras}}]
Let $H$ be a Hopf algebra. An element $t \in H$ is a \emph{left integral} if $ht=e(h)t$ for all $h \in H$. Left integrals form a $\kay$-submodule of $H$ denoted by $\int_H$. If $H/\kay$ is finite,  $\dim_{\kay} \int_H =1$. 
\end{theorem}

\emph{We take $\Tr_G$ to be a $\kay$-generator of $\int_{\sO(G)^{\vee}}$, which is unique up to scaling by elements of $\kay^\times$.} If there is $t \in \int_{\sO(G)^{\vee}}$ such that $t(1) \neq 0$, we normalize to have $\Tr_{G}(1)=1$. Maschke's Theorem establishes that this amounts to $G$ being linearly reductive; see \cite[2.2.1, 2.4.6]{MontgomeryHopfAlgebras}.

\begin{remark}[Geometric description of integrals]
Let $G/\kay$ be a group-scheme. A \emph{Hopf module} on $G$ is a quasi-coherent $\sO_G$-module $\sF$ endowed with an equivariant $\sO_G$-module structure. More precisely, $\sF$ is a quasi-coherent (right) $\sO_G$-module equipped with an $\sO_G$-linear (right) coaction $\rho \mathrel{:} \sF \to \sF \otimes \sO_G$, where the (right) $\sO_G$-linear structure of $\sF \otimes \sO_G$ is
\[
(\sF \otimes \sO_G) \otimes \sO_G \xrightarrow{\id \otimes \nabla} (\sF \otimes \sO_G) \otimes (\sO_G \otimes \sO_G) \xrightarrow{\id \otimes \tau \otimes \id} (\sF \otimes \sO_G) \otimes (\sO_G \otimes \sO_G) \xrightarrow{c\otimes \Delta} \sF\otimes \sO_G, 
\]
where $\id \otimes \tau \otimes \id$ is the isomorphism swapping the two middle tensor factors of $(\sF \otimes \sO_G) \otimes (\sO_G \otimes \sO_G)$, and $c \mathrel{:} \sF \otimes \sO_G \to \sF$ is the $\sO_G$-linear structure of $\sF$.  In other words,  $\sF$ is equipped with a structural map realizing it as a right $\sO_G$-comodule \cite[\S1.6]{MontgomeryHopfAlgebras}. One defines the category of Hopf modules in the obvious way; see \cite[\S 1.9]{MontgomeryHopfAlgebras}.

The Fundamental Theorem of Hopf modules \cite[\S1.9]{MontgomeryHopfAlgebras},  establishes that,  for every coherent Hopf module $\sF$ on a finite algebraic group $G/\kay$, there is a canonical isomorphism of Hopf modules $u^* \sF^{G} \to \sF$, where $\sF^G \coloneqq \{x \in \sF(G) \mid \rho(x) = x\otimes 1 \}$ is the $\kay$-submodule of (co)invariants of $\sF$ under the action of $G$.  The Hopf module structure of $u^* \sF^G$ is the one inherited from $\sO_G$. That is,  the coaction $u^* \sF^G \to u^* \sF^G \otimes u^*\sF^G$ is the tensor product of $\nabla \mathrel{:} \sO_G \to \sO_G \otimes \sO_G$ by $\sF^G$. By definition,  $\sF^{G}$ is a $\kay$-submodule of $u_* \sF$,  say $j \mathrel{:} \sF^G \to u_* \sF$.  Thus,  the canonical map $u^* \sF^{G} \to \sF$ is the composition $u^* \sF^ G \xrightarrow{u^*j} u^* u_* \sF \xrightarrow{\mathrm{can}} \sF$.  In particular,  since $u \circ e = \id$,  pulling back the canonical isomorphism $u^*\sF^G \to \sF$ along $e \mathrel{:} \ast \to G$ gives a canonical isomorphism of $\kay$-modules $\sF^G \to e^* \sF$. 

The canonical module $\omega_G = u^! \omega_{\kay}$ is a major example of a Hopf module. Indeed, $\omega_G$ (or rather its global sections) is a Hopf $\kay$-algebra, for it is the dual Hopf algebra of $\sO(G)$. In particular, there is a product morphism $\omega_G \otimes \omega_G \to \omega_G$ which can be conveniently rewritten as $\sO(G)^{\vee} \otimes \omega_G \to \omega_G$.  This gives a (left) action of $\sO(G)^{\vee} $ on $\omega_G$,  which induces a (right) coaction of $\sO(G)$ on $\omega_G$ (see \autoref{rem.ActionsCoactions}) giving the desired coaction map $\omega_G \to \omega_G \otimes \sO_G$.  The fact that this coaction is $\sO_G$-linear is the content of \cite[2.1.4]{MontgomeryHopfAlgebras}.

When the Fundamental Theorem of Hopf modules is applied to the Hopf module $\omega_G$, we obtain a canonical isomorphism $u^* (\omega_G)^G \to \omega_G$. On the other hand, it is an immediate consequence of the definitions that $(\omega_G)^G = \int_{\sO(G)^{\vee}}$. However, $\omega_G \cong \sO_G$, for $G$ is Gorenstein. Therefore, the $\kay$-module $\int_{\sO(G)^{\vee}}$ is $1$-dimensional and every free $\kay$-generator of $\int_{\sO(G)^{\vee}}$ is a free $\sO(G)$-generator of $\Hom_{\kay}(\sO(G), \kay)$. This is how \cite[2.1.3]{MontgomeryHopfAlgebras} works.
\end{remark}

\begin{remark} \label{rem.actionoftraceA}
 $\Tr_G$ is nonsingular \cite[Theorem 2.1.3]{MontgomeryHopfAlgebras}.
\end{remark}

\begin{remark} \label{rem.actionoftraceB} To be consistent with the forthcoming discussion,  we ought to define $\Tr_G$ as the $\kay$-linear map $\gamma \mapsto t \cdot \gamma$ (for a choice  $0 \neq t \in \int_{\sO(G)^{\vee}}$)
 rather than $\gamma \mapsto t(\gamma)$.  Nonetheless,  we get the same result either way as 
 \begin{align*}
 g(t \cdot \gamma ) = g((\textnormal{id} \otimes t)(\nabla(\gamma)))=(g \otimes t)(\nabla(\gamma))=(g\cdot t) (\gamma)=\left(g(1)t\right) (\gamma)&=g(1)t(\gamma)=g(t(\gamma)),
 \end{align*}
 for all $g \in \sO(G)^{\vee}$. In other words, the following diagram is commutative 
\[
\xymatrix{
\sO(G) \ar[r]^-{\nabla} \ar[d]_-{\Tr_{G}} & \sO(G) \otimes \sO(G) \ar[d]^-{ \textnormal{id} \otimes\Tr_{G}}\\
\kay \ar[r]^-{u} & \sO(G)
}
\]
\end{remark}

\begin{example} \label{ex.ExampleTraces}
Let $G$ be a finite discrete group. A left integral $t$ of $H\coloneqq\Hom_{\mathsf{Set}}(G,\kay)$ must satisfy $\gamma t = \gamma(1)t$ for all $\gamma \in H$. That is, $\gamma(g)  t(g) = \gamma(1)t(g)$ for all $\gamma \in H$, $g \in G$. Therefore, $\int_H=\kay \cdot \varepsilon$, where $\varepsilon(g)=0$ for all $g\neq 1$ and $\varepsilon(1)=1$. A left integral $t$ of $\kay[G]$ is characterized by $g t = t$ for all $g \in G$. For example, $t=\sum_{g \in G} g$ is a left integral and by uniqueness $\int_{\kay[G]}=\kay\cdot t$. Thus, for the constant group-scheme $G$, a trace $\Tr_G\mathrel{:}\sO(G) \to \kay$ is given by $\gamma \mapsto \sum_{g \in G}{\gamma(g)}$, which is the Reynolds operator. Note that $\Tr_G(1) = o(G)$. If $p \nmid o(G)$, we divide by $o(G)$ so that $\Tr_G(1) = 1$. For the diagonalizable group $D(G)$, a trace $\Tr_{D(G)}\mathrel{:} \kay[G] \to \kay$ is the projection onto the direct $\kay$-summand generated by $1$, so $\Tr_{\kay[G]}(1)=1$. For $\alpha_{p^e}$, a left integral $t$ has to satisfy $\xi^i \cdot t = 0$ for all $1 \leq i \leq p^{e}-1 $. For instance $t=\xi^{p^e-1}$, and so $\int_{\alpha_{p^e}}=\kay \cdot \xi^{p^{e}-1}$. Hence, a trace for $\alpha_{p^e}$ is obtained by projecting onto the direct $\kay$-summand generated by $\xi^{p^e-1}$. In particular, $\Tr_{\alpha_{p^e}}(1)=0$.
\end{example}

We work out the affine case next. Let $S$ be a $\kay$-algebra endowed with an action $\alpha\mathrel{:} S \to \sO(G_S)$ of $G$ on $\Spec S$; set $R = S^{G}$. Consider the $S$-linear map $\Tr_{G_S}\coloneqq \textnormal{id} \otimes \Tr_{G}\mathrel{:} \sO(G_S) \to S$ where the $S$-algebra structure of $\sO(G_S)$ is $u_S\mathrel{:}S \to \sO(G_S)$. Note that $\alpha \mathrel{:} S \to \sO(G_S)$ is an $R$-linear map as it is a ring homomorphism and $R=S^G$. We then get an $R$-linear map $\Tr_{G_S} \circ \alpha \mathrel{:} S \to S$, which is none other than $s \mapsto \Tr_G \cdot s$ (using the induced action of $\sO(G)^{\vee}$ on $S$). 

\begin{definition-proposition} \label{TraceCodomain} The $R$-linear map $\Tr_{G_S} \circ \alpha \mathrel{:} S \to S$ has image in $R$. One then defines $\Tr_{S/R} \mathrel{:} S \to R$ to be the restriction of the codomain.
\end{definition-proposition}
\begin{proof}
Recall that $\Tr_{G_S} \circ \alpha \mathrel{:} s \mapsto \Tr_G \cdot s$. One readily verifies that $\Tr_G \cdot s$ is an invariant under the action of $\sO(G)^{\vee}$ on $S$ and so it is a coinvariant under the coaction of $\sO(G)$; see \autoref{rem.ActionsCoactions}. However, we present a direct,  conceptual proof.  Consider the following diagram:
\[
\xymatrixrowsep{3pc}\xymatrixcolsep{6pc}\xymatrix{
S \ar[r]^-{\alpha} & \sO(G_S) \ar@/^/[r]^{\alpha \otimes \textnormal{id}} \ar@/_/[r]_{\nabla_S} \ar[d]_-{\Tr_{G_S}} & \sO(G_S) \otimes_S \sO(G_S) \ar[d]^-{\textnormal{id} \otimes \Tr_{G_S}}\\
R \ar[r]^-{\subset}
 & S \ar@/^/[r]^{\alpha}\ar@/_/[r]_{u_S} & \sO(G_S)
}
\]
The bottom sequence is exact by definition. The top sequence, although not necessarily exact, satisfies $\alpha(S)\subset \ker(\alpha \otimes \textnormal{id}, \nabla_S)$ according to first axiom of $\alpha$ being a coaction. Thus, it suffices to prove that the following two squares are commutative
\[
\xymatrixcolsep{3pc}\xymatrix{
\sO(G_S) \ar[r]^-{\alpha \otimes \id} \ar[d]_-{\Tr_{G_S}} & \sO(G_S) \otimes_S \sO(G_S) \ar[d]^-{\id \otimes \Tr_{G_S}}\\
S \ar[r]^-{\alpha} & \sO(G_S)
} \qquad
\xymatrix{
\sO(G_S) \ar[r]^-{\nabla_S} \ar[d]_-{\Tr_{G_S}} & \sO(G_S) \otimes_S \sO(G_S) \ar[d]^-{\id \otimes\Tr_{G_S}}\\
S \ar[r]^-{u_S} & \sO(G_S)
}
\]
The commutativity of the first square is fairly straightforward. The commutativity of the second one follows from base changing by $S$ the commutative square in \autoref{rem.actionoftraceB}.
\end{proof}
\begin{remark}
The trace map $\Tr_{S/R} \mathrel{:} S \to R$ depends on the action of $G$ on $S$ as well as on the choice of $\Tr_G \mathrel{:} \sO(G) \to \kay$. Our notation is misleading as it does not reflect that. We hope this ambiguity will not lead to confusion but would rather improve the readability. 
\end{remark}

For a general finite $G$-quotient $q\mathrel{:} Y \to X$, we may define an $\sO_X$-linear map $\Tr_{Y/X} \mathrel{:} q_* \sO_Y \to \sO_X$ by gluing the trace maps previously constructed on affine charts; see \autoref{quotients}. To show that trace maps can be glued together, we have:
\begin{proposition}
Let $R \subset S$ be a finite $G$-quotient and $W \subset R$ be a multiplicatively closed set. Then, $\Tr_{W^{-1}S/W^{-1}R}$ and $W^{-1}\Tr_{S/R}$ are the same $W^{-1}R$-linear map $W^{-1} S \to W^{-1}R$.
\end{proposition}
\begin{proof}
The localization of $\alpha \mathrel{:} S \to \sO(G_S)$ induces a coaction $W^{-1}\alpha\mathrel{:} W^{-1} S \to W^{-1}S \otimes \sO(G)$ of $\sO(G)$ on $W^{-1}S$; given by $s/u \mapsto 1/u \otimes \alpha(s)$. Further, the corresponding ring of coinvariants is $W^{-1}R$. It follows that $\Tr_{W^{-1}S/W^{-1}R}=W^{-1}\Tr_{S/R}$.
\end{proof}
To define $\Tr_{Y/X} \mathrel{:} q_* \sO_Y \to \sO_X$, we may proceed directly as follows. Let $\Tr_{G_Y} \mathrel{:} {p_1}_{*}\sO_{G_Y} \to \sO_Y$ be the pullback of $\Tr_G : \sO(G) \to \kay$ along $Y/\kay$, where $p_1: G_Y = Y \times G \to Y$ is the canonical projection. Apply $q_*$ to $\alpha \coloneqq a^{\#}\mathrel{:}\sO_Y \to a_* \sO_{G_Y}$ (recall that $a \mathrel{:} G_Y = Y \times G \to Y$ is the action of $G$ on $Y$) to get a morphism $q_* \alpha \mathrel{:} q_* \sO_Y \to q_* a_* \sO_{G_Y} \cong q_* {p_1}_* \sO_{G_Y}$. Composing $q_* \alpha \mathrel{:} q_*\sO_Y \to q_* p_1 \sO_{G_Y}$ with $q_*\Tr_{G_Y} \mathrel{:} q_* {p_1}_* \sO_{G_Y} \to q_* \sO_Y$ yields an $\sO_X$-linear map $q_* \sO_Y \to q_* \sO_Y$. By applying \autoref{TraceCodomain} locally, $q_* \sO_Y \to q_* \sO_Y$ restricts to a morphism $q_* \sO_Y \to (q_* \sO_Y)^{G}$. Using the canonical isomorphism $\sO_X \xrightarrow{} (q_* \sO_Y)^{G}$ gives the desired trace map $\Tr_{Y/X} \mathrel{:} q_*\sO_Y \to \sO_X$.

\subsection{Initial properties of the trace}  Let $L/\kay$ be a finite field extension and $G$ be a finite group acting (on the left) by $\kay$-endomorphisms on $L$. Letting $K\subset L$ denote the fixed subfield, then $L/K$ is Galois with Galois group $G$ if and only if $G$ acts faithfully on $L$. This is equivalent to $\Tr_{L/K}$ being nonsingular and having $[L:K]=o(G)$. We obtain the following analog of this principle as the main result of this section.
\begin{theorem} \label{thm.Torsor_and_trace}
A finite $G$-quotient $q\mathrel{:}Y \to X$ is a $G$-torsor if and only if $q_* \sO_Y$ is locally free of rank $o(G)$ and $\Tr_{Y/X}$ is nonsingular.
\end{theorem}
\begin{proof}
The statements are local on $X$ and so we may work on an affine chart $R\coloneqq S^{G}\subset S$. Consider the fibered coproduct diagram
\[
\xymatrix{
\sO(G_S)=S \otimes \sO(G)   \\
&S \otimes_R S \ar@{.>}[ul]|-{\varphi \coloneqq u_S \otimes \alpha} & S \ar[l]^-{p_2} \ar@/_/[ull]_-{\alpha} \\
&S \ar[u]_-{p_1} \ar@/^/[uul]^-{u_S}  &R \ar[l] \ar[u]}
\]
By definition, $R \subset S$ is a $G$-torsor if and only if it is a fppf map and $\varphi\coloneqq u_S \otimes \alpha$ is an isomorphism. In this case, $R \subset S$ is locally free of rank $o\coloneqq o(G)$. Thus, we may assume this and prove that $\varphi$ is an isomorphism if and only if $\Tr_{S/R}$ is nonsingular. Note that both statements are local on $R$ and that for all $\p \in \Spec R$, we have $\varphi_{\mathfrak{p}}=u_{S_{\mathfrak{p}}}\otimes \alpha_{\mathfrak{p}}$, and $\Tr_{S_{\mathfrak{p}}/R_{\mathfrak{p}}}=\Tr_{S/R} \otimes_R R_{\mathfrak{p}}$ by \autoref{TraceCodomain}. Then, we may assume that $R$ is local and so that $R \subset S$ is a semi-local free extension. 

To show $\varphi\mathrel{:} S\otimes_R S \to \sO(G_S)$ is an isomorphism, it suffices to do it when considered as an $S$-linear map, where the $S$-linear structures are given by $p_1$ and $u_S$; respectively. Let $s_1,\ldots ,s_o$ be a basis of $S/R$ and $\gamma_1,\ldots ,\gamma_o$ be one of $\sO(G)/\kay$. Thus, $1\otimes s_1,\ldots ,1\otimes s_o $ is a basis for $p_1\mathrel{:} S \to S \otimes_R S$ and $1 \otimes \gamma_1,\ldots ,1 \otimes \gamma_o $ is a basis for $u_S\mathrel{:} S \to \sO(G_S)$. Next, we set:
\[
\varphi(1 \otimes s_i)=(u_S \otimes \alpha)(1 \otimes s_i)= \alpha(s_i) = \sum_{m=1}^{o}{ a_i^m \otimes \gamma_m}.
\]
That is,  $M\coloneqq(a_i^m)_{m,i}$ is the $S$-matrix representation
of $\varphi$ in these bases. Thus, $\varphi$ is an isomorphism if and
only if $M$ is nonsingular.  We describe next the symmetric $R$-matrix
associated to $\Tr_{S/R}$ as an $R$-bilinear form on $S$, in terms of
$M$.  Let $T\coloneqq(\Tr_G(\gamma_m \gamma_n))_{m,n}$ be the
$\kay$-matrix representing $\Tr_G$ as a $\kay$-bilinear form on
$\sO(G)$, in the $\kay$-basis $\gamma_1,\ldots ,\gamma_o$.  Recall that
$T$ is nonsingular (as well as symmetric) by
\autoref{rem.actionoftraceA}.  We then have:

\begin{claim*} 
$M^{\top}TM$ is the matrix representation of $\Tr_{S/R}$ as an $R$-bilinear form on $S$ in the $R$-basis $s_1,\ldots ,s_o$.
\end{claim*}
\begin{proof}[Proof of the claim] This amounts to the following computation:
\begin{align*}
  \Tr_{S/R}(s_i \cdot s_j)
  &=\Tr_{G_S} (\alpha(s_i \cdot s_j))
    = \Tr_{G_S} (\alpha(s_i) \cdot \alpha ( s_j)) \\
  &= \Tr_{G_S} \left( \left(\sum_{m=1}^{o}{a_i^m \otimes \gamma_m}\right) \left( \sum_{n=1}^{o}{a_j^n\otimes \gamma_n}\right) \right)\\
  &= \Tr_{G_S} \left( \sum_{1\leq m,n \leq o}{a_i^m a_j^n\otimes \gamma_m \gamma_n} \right) \\
  &=  \sum_{1\leq m,n \leq o}{ a_i^m a_j^n} \cdot \Tr_{G}(\gamma_m \gamma_n) \\
  &= \sum_{1\leq m,n \leq o}{  a_i^m T_{mn} a_j^n} = (M^{\top}TM)_{ij}. 
\end{align*}
\end{proof}
Then, $\Tr_{S/R}$ is nonsingular if and only if so is $M$, for 
\[
  \disc \Tr_{S/R}= \det T \cdot (\det M)^2=\disc \Tr_G \cdot (\det
  M)^2.
\]
\end{proof}

\begin{scholium}[{\cf \cite[8.3.1]{MontgomeryHopfAlgebras},  \cite{KreimerTakeuchi}}] \label{sch.scholium}
With notation as in the proof of \autoref{thm.Torsor_and_trace}, suppose $R\subset S$ is locally free but $\varphi$ is only surjective. Then, $\Tr_{S/R}$ is nondegenerate. If $S$ is further a domain, then $\varphi$ is an isomorphism and so $\Tr_{S/R}$ is nonsingular.  
\begin{proof}
Let $d$ be the rank of $R \subset S$. The surjectivity of $\varphi$ implies $d \geq o$. Hence, the matrix $M$ defines a surjective $S$-linear map $S^{\oplus d} \to S^{\oplus o}$. Therefore, the $S$-linear transformation $M^{\top}: S^{\oplus o} \to S^{\oplus d}$ is injective as it is the $S$-dual of $M$. We claim that $M^{\top}T M$ defines an injective $R$-linear operator $R^{\oplus o} \to R^{\oplus o}$. Since $M^{\top}$ and $T$ are injective, it remains to explain why if $M \cdot \vec{v} = 0$ for a column vector $\vec{v} \in R^{\oplus o}$ then $\vec{v} = 0$. This is just another way to say that $\alpha$ is injective: if $\vec{v}=(r_1,\ldots ,r_d)^{\top}$ and $M \cdot \vec{v} = (t_1,\ldots ,t_o)^{\top}$, then $\alpha(r_1 s_1 + \cdots +r_d s_d) = t_1 \otimes \gamma_1 + \cdots + t_o \otimes \gamma_o $. However, $\alpha$ is injective for  $\id_S=e_S\circ \alpha$---the second action axiom. In conclusion, the determinant of $M^{\top} T M$ is a nonzerodivisor on $R$ and so $\Tr_{S/R}$ is nondegenerate. If $S$ is further a domain, the determinant of $M^{\top} T M$ is also a nonzerodivisor on $S$. Therefore, $M^{\top} T M$ defines an injective $S$-linear operator $S^{\oplus o} \to S^{\oplus o}$ forcing $M$ to be injective and so $\varphi$ to be an isomorphism.
\end{proof}
\end{scholium}

\begin{corollary}[{\cf \cite[VI, Theorem 6.8]{AltmanKleimanIntroToGrothendieckDuality}}] \label{DiscriminantCorollary}
Let $q\mathrel{:} Y \to X$ be a finite $G$-quotient. If $q$ is a $G$-quasitorsor\footnote{That is, $q_x$ is a $G$-torsor under the localized action for all codimension $1$ point $x\in X$; see \autoref{term.QuasiTorsor}} and flat, then it is a $G$-torsor. 
\end{corollary}
\begin{proof}
Since $q$ is flat, the sheaf of principal ideals $\disc \Tr_{Y/X} \subset \sO_X$ cuts out the locus of points $x \in X$ where $q_x \mathrel{:}Y \times_X \Spec \sO_{X,x} \to \Spec \sO_{X,x}$ is not a $G$-torsor under the induced action. Thus, if $q$ is not a $G$-torsor then it fails to be so in pure codimension $1$.  
\end{proof}

\begin{corollary}[{\cf \cite[I, Proposition 3.8]{MilneEtaleCohomology}}] \label{cor.OpenNatureTorsoness}
Let $q\mathrel{:} Y \to X$ be a finite $G$-quotient. The locus $W \subset X$ of points $x \in X$ where $q_x\mathrel{:}Y \times_X \Spec \sO_{X,x} \to \Spec \sO_{X,x}$ is a $G$-torsor is open.
\end{corollary}
\begin{proof}
Let $x \in W$. There is an open neighborhood $W' \ni x$ such that $q_{x'}$ is faithfully flat for all $x' \in W'$. Hence, the open $W' \setminus V(\disc \Tr_{q^{-1}W'/W'}) \ni x$ is contained in $W$. 
\end{proof}

\begin{corollary} \label{cor.Tr_generates} Let $R \subset S$ be a finite $G$-quotient. If $R \subset S$ is a $G$-torsor,  then $\Tr_{S/R}$ freely generates $\Hom_R(S,R)$ as an $S$-module. The same holds if $R\subset S$ is a $G$-quasitorsor and $R$, $S$ satisfy $\mathbf{S}_2$.
\end{corollary}
\begin{proof}
The first statement is a rephrasing of what it means that $\Tr_{S/R}$ is nonsingular, \ie the $S$-linear map $S \to \Hom_R(S,R)$ given by $s \mapsto \Tr_{S/R}(s \cdot -)$ is an isomorphism. For the second statement, note that $R \subset S$ is finite under those hypothesis (as the torsor locus is a big open). Then, $S$ and $\Hom_R(S, R)$ are both $\mathbf{S}_2$ $R$-modules. Indeed, $S$ is $\mathbf{S}_2$ as an $R$-module as depth is invariant under finite restriction of scalars. For the $\mathbf{S}_2$-ness of $\Hom_R(S, R)$, see \cite[\href{https://stacks.math.columbia.edu/tag/0AV6}{Tag 0AV6}]{stacks-project}. Hence, to check that the aforementioned map $S \to \Hom_R(S,R)$ is an isomorphism, it suffices to do it in codimension $1$ on $\Spec R$. This is the case if $R \subset S$ is a $G$-quasitorsor.
\end{proof}

\subsection{Cohomological tameness and total integrals}
We are ready to formulate the notion of \emph{tameness} our covers will have. In \cite{KerzSchmidtOnDifferentNotionsOfTameness}, several notions of tameness are considered yet \emph{cohomological tameness} resulted to be the strongest one. Such condition is going to be imposed by \emph{splinters}, which are those rings splitting off from any finite extension. As a matter of fact, strongly $F$-regular rings are splinters. See \cite{MaFSplittings,HochsterContractedIdealsFromIntegralExtensions}. 

\begin{definition}[\cite{ChinburgErezPappasTaylorTameActions,KerzSchmidtOnDifferentNotionsOfTameness}]
A finite $G$-quotient $S^G\subset S$ is \emph{(cohomologically) tame} if $\Tr_{S/S^G}$ is surjective.
\end{definition}

\begin{remark} By Maschke's Theorem, a finite $G$-quotient with $G$ linearly reductive is tame. Indeed,  $\Tr_G (1) = 1$ and so $\Tr_{S/S^G}(1) = \Tr_{G_S}\left(\alpha(1)\right) = \Tr_G(1) = 1$. In contrast, if $G$ is unipotent, $S^G \subset S$ is tame only if it is a trivial torsor \cite[Proposition 6.2]{ChinburgErezPappasTaylorTameActions}.
\end{remark}

\begin{remark}
In the Hopf algebras literature, the surjectivity of $\Tr_{S/S^G}$ is referred to as the existence of \emph{total integrals} for the right $\sO(G)$-comodule $S$. To the best of the author's knowledge, total integrals were introduced by Y.~Doi \cite{DoiAlgebraswithTotalIntegrals} yet they were formulated in slightly different terms. The equivalence between the existence of Doi's total integrals and the surjectivity of the trace was established in \cite{CohenFischmanSemisimpleExtensions}. See \cite[\S 4.3]{MontgomeryHopfAlgebras}.
\end{remark}

\begin{proposition} \label{pro:SplinterImpliesTame}
Let $R \subset S$ be a finite $G$-quotient. Suppose that $R$ is a splinter (and so normal) and that $S$ satisfies $\mathbf{S}_2$. If $R \subset S$ is a $G$-quasitorsor then it is tame.
\end{proposition}
\begin{proof}
By \autoref{cor.Tr_generates}, $\Tr_{S/R} \cdot S \cong \Hom_R(S,R)$ as $S$-modules.  Since $R$ is a splinter,  there is a splitting $\Tr_{S/R} \cdot s \mathrel{:} S \to R$ of $R \subset S$.  That is,  $\Tr_{S/R} \cdot s \mathrel{:} 1\mapsto 1 $ and so $\Tr_{S/R} \mathrel{:} s \mapsto 1$. 
\end{proof}

\section{On the existence of a maximal cover} \label{ExistenceofMaximalCover}

We come now to \autoref{que:MainQiestion}. \emph{We work in \autoref{setup}.} Suppose that the restriction of torsors pointed map $
\varrho^1_{X,U}(G) \mathrel{:} \check{H}^1(X_{\textnormal{fl}}, G) \to \check{H}^1(U_{\textnormal{fl}}, G) $ is not surjective for some finite algebraic group $G/\kay$. This means that there is a $G$-torsor $V \to U$ that cannot be extended to a $G$-torsor over $X$. Our first task is to convert this into a local algebra setting so that \autoref{thm.TransformationRule} can be applied. 

\subsection{Local finite torsors}\label{TheQuestions}
In this section, we explain why, if $\varrho^1_{X,U}(G)$ is not surjective, there exist a finite algebraic group $G'/\kay$ with $(G')^{\circ} = G^{\circ}$ and a local finite $G'$-quotient $(R, \mmm, \kay, K) \subset (S, \nnn, \kay, L)$ that is a $G'$-torsor over $U$ but not everywhere and $S$ satisfies $\mathbf{S}_2$. This is done by taking integral closures. Let $h\mathrel{:} V\to U$ be a finite morphism. Its \emph{integral closure} is the finite morphism $\tilde{h}\mathrel{:}Y \to X$ where $Y= \Spec S$ and $S\coloneqq H^0(\sO_V,V)$. Taking integral closures is functorial on finite $U$-schemes and the pullback of $\tilde{h}$ to $U$ recovers $h$. Also, if $h:V \to U$ is the restriction of a $G$-torsor $Y \to X$, then $Y \to X$ must be the integral closure of $h$.

\begin{lemma} \label{prop.extensionOfTheAction}
Let $h\mathrel{:} V\to U$ be a finite $G$-torsor. Then, the action $a\mathrel{:} V \times G \to V$ extends to a unique action $\tilde{a}\mathrel{:} Y \times G \to Y$ such that $\tilde{h}\mathrel{:}Y \to X$ is its finite quotient morphism. Moreover, $S=H^0(Y, \sO_Y)$ is an $\mathbf{S}_2$ semi-local ring. 
\end{lemma}

\begin{proof}
We show that the coaction of $a$ on global sections gives $\tilde{a}$. Note that $R=H^0(U,\sO_U)$ as $R$ is $\mathbf{S}_2$. Further, $S=H^0(V,\sO_V)$ is an $\mathbf{S}_2$ ring too. Indeed, since $h\mathrel{:}V \to U$ is a faithfully flat finite morphism, $h_* \sO_V$ is an $\mathbf{S}_2$ coherent $\sO_U$-module. Then, $i_* h_* \sO_V$ is an $\mathbf{S}_2$ coherent $\sO_X$-module by \cite[Theorem 1.12]{HartshorneGeneralizedDivisorsOnGorensteinSchemes}, where $i \mathrel{:} U \to X$ is the immersion.  However, $H^0(X, i_* h_* \sO_V)=S$ by definition. Since $S$ is an $\mathbf{S}_2$ $R$-module, it is an $\mathbf{S}_2$ ring (finite restriction of scalars does not change depth). Likewise, $S \otimes \sO(G)$ is also an $\mathbf{S}_2$ ring. Since $V \times G \subset Y \times G$ is a big open, $H^0(V\times G,\sO_{V\times G}) = S \otimes \sO(G)$. Thus, the coaction of $a$ on global sections induces a coaction $a(V)\mathrel{:} S \to S \otimes \sO(G)$, which defines the desired action $\Tilde{a}$. Furthermore, the corresponding ring of invariants is $H^0(V,\sO_V)^{G}=(q_* \sO_V)^{G}(U)  = \sO_U(U) = R$.
\end{proof}

\begin{remark} Let $R=S^{G} \subset S$ be a $G$-quasitorsor. As in \autoref{cor.OpenNatureTorsoness}, let $W \subset X$ be the big open over which $Y/X$ is a torsor. Since $R$ is $\mathbf{S}_2$, then $R=H^0(W,\sO_W)$. However,  $S \to H^0(Y_W, \sO_{Y_W})$ may not be an isomorphism unless $S$ was $\mathbf{S}_2$ to start with.  
\end{remark}

In \autoref{prop.extensionOfTheAction}, we say that $h$ \emph{extends across its integral closure} if $\tilde{h}$ is a $G$-torsor. Our problem then reduces to study the extent to which finite $G$-torsors $h \mathrel{:} V \to U$ extend across their integral closure. Next, we may assume $V$ connected and so $S$ local (as $R$ is henselian):
\begin{lemma} \label{lem.dominatingbyconnected}
Let $h\mathrel{:}V\to U$ be a finite $G$-torsor. There exist a $G'$-torsor $h'\mathrel{:}V'\to U$ and an equivariant finite $U$-morphism $f\mathrel{:} V' \to V$ such that $V'$ is connected and $(G')^{\circ} \cong G^{\circ}$.\footnote{Equivariant means the existence of a homomorphism $\varphi \mathrel{:} G' \to G$ so that $a \circ (f \times \varphi) = f \circ a'$, where $a$, $a'$ are the corresponding actions.}
\end{lemma}
\begin{proof}
 Consider $G=G^{\circ} \rtimes \pi_0(G)$. The image of $h$ under
$\check{H}^1(U_{\textnormal{fl}},G) \to
\check{H}^1\bigr(U_{\textnormal{fl}},\pi_0(G)\bigl)$ 
gives a $\pi_0(G)$-torsor $\bar{h}\mathrel{:} W \to U$ and
a $G^{\circ}$-torsor $V \to W$ factoring 
$h \mathrel{:} V \to W \to U$; see
\autoref{rem:GaloisCorrespondence}. Let $W'$ be the Galois closure of
any of the connected components of $W$; see \cite[Lemma
4.4.1.8]{MurreLecturesFundamentalGroups}. In particular, $W' \to U$ is
a finite Galois cover, \ie an \'etale connected
$\textnormal{Gal}(W'/U)$-torsor.  Set $V' \coloneqq V \times_W W'$. By
\cite[II, Lemma 1]{NoriFundamentalGroupScheme} and \cite[Proposition
2.2]{EsnaultViehwegSurfaceSingularitiesDominatedSmoothVarieties}, 
$V' \to U$ is a $G'$-torsor, where
$G\coloneqq G \times_{\pi_0(G)} \textnormal{Gal}(W'/U)$,
as $X$ is integral and $x \in X(\kay)$.\footnote{By
  \cite{NoriFundamentalGroupScheme}, if $V_i \to U$ are finite
  $G_i$-torsors ($i=0,1,2$) and $f_i \mathrel{:} V_i \to V_0$
  ($i=1,2$) are equivariant maps, then $V_1 \times_{V_0} V_2$ is a
  $G_1 \times_{G_0} G_2$-torsor over $U$ provided that $U$ is
  integral and $U(\kay) \neq \emptyset$. In our case, $U$ is integral
  but $U(\kay) = \emptyset$. This is remedied in
  \cite{EsnaultViehwegSurfaceSingularitiesDominatedSmoothVarieties}
  by using that $U \subset X$ and $X(\kay) \neq \emptyset$.} Note
that $(G')^{\circ} \cong G^{\circ}$ as $G=G^{\circ} \times \pi_0(G)$
as schemes.  It remains to explain why $V'$ is connected.  
\begin{claim}
  \label{ConnectednessForTorsors}
  If $X_2 \to X_1$ is a finite $G_0$-torsor with $X_1$ and $G_0$
  connected, then $X_2$ is connected. 
\end{claim}
\begin{proof}[Proof of the claim] 
  The following argument was kindly provided to the author by one of the anonymous referees. It is well-known that $X_2 \to X_1$ is a (universal) geometric quotient; see \cite[Theorem 4.16]{AbelianVarietiesGeerMooen} for details. Further, since $G_0$ is homeomorphic to a point, it follows that $X_2 \to X_1$ is a homeomorphism. The claim then follows.
\end{proof}
The lemma follows by noting that $V' \to W'$ is a $G^{\circ}$-torsor.
\end{proof}

In \autoref{lem.dominatingbyconnected}, if $h'\mathrel{:} V' \to U$ extends across its integral closure then so does $h\mathrel{:}V \to U$; see \cite[Proposition 2.3]{EsnaultViehwegSurfaceSingularitiesDominatedSmoothVarieties}. Summing up:

\begin{proposition}\label{pro.TechnicalProposition}
Let $h \mathrel{:} V \to U$ be a $G$-torsor that is not the restriction of a $G$-torsor over $X$. Then, there exists a local finite $G'$-quotient $Y=\Spec (S,\nnn, \kay) \to X$ that restricts to a torsor over $U$ but not over $X$. Moreover, $S$ is an $\mathbf{S}_2$ ring and $(G')^{\circ}\cong G^{\circ}$.
\end{proposition}
\begin{remark} The residue fields are the same because $\kay$ is algebraically closed. This is to ensure that every cover is endowed with a $\kay$-rational point lying over $x$. 
\end{remark}
\begin{remark} \label{AlternateReduction} 
Using a slightly different argument,  we can obtain the weakening of \autoref{lem.dominatingbyconnected} (and so of \autoref{pro.TechnicalProposition}) where $(G')^{\circ}$ is either trivial or isomorphic to $G^{\circ}$.  Let $\bar{h}$ be as in the proof of \autoref{lem.dominatingbyconnected}.  If $\bar{h}$ is not trivial, the result follows using \cite[Lemma 4.4.1.8]{MurreLecturesFundamentalGroups} (see \cite[\S2.4]{CarvajalSchwedeTuckerEtaleFundFsignature}), in which case $(G')^{\circ}=*$. If $\bar{h}$ is trivial, by exactness of $\check{H}^1(U_{\textnormal{fl}},G^{\circ}) \to \check{H}^1(U_{\textnormal{fl}},G) \to \check{H}^1\bigr(U_{\textnormal{fl}},\pi_0(G)\bigl)$, 
$h$ is isomorphic to $V' \wedge^{G} G^{\circ} \to U$ for some $G^{\circ}$-torsor $V' \to U$, which is nontrivial as so is $h$.  By \autoref{ConnectednessForTorsors}, $V'$ is connected.   Further,  the canonical $U$-morphism $V' \to V' \wedge^{G} G^{\circ}$ is equivariant; see \cite[III, 1.3.6.(i)]{GiraudNonabelianCohomology}. 
\end{remark}

\subsection{The generalized transformation rule for the $F$-signature} By the transformation rule in \cite{CarvajalSchwedeTuckerEtaleFundFsignature}, the size of $s(R) \in [0,1]$---the $F$-signature of $R$---imposes constraints on the existence and size of local extensions $R \subset S'$ in \autoref{pro.TechnicalProposition} if $G'/\kay$ is \'etale. We abstract here the properties we require for this transformation rule to exist in further generality. 

\begin{theorem}\label{thm.TransformationRule}
Let $(A, \aaa) \subset (B, \bbb)$ be a finite local extension. Suppose that there exists $T \in \Hom_A(B,A)$ such that: $T$ is a free generator of $\Hom_A(B,A)$ as a $B$-module, $T$ is surjective, and $T(\bbb) \subset \aaa$. If $A$ is not a domain, then $s(A)=0=s(B)$. Else, the formula 
\[
\left[\kappa(\bbb):\kappa(\aaa)\right]  \cdot s(B)= \dim_{K(A)} {B_{K(A)}} \cdot s(A).\]
holds. In particular, $B$ is a strongly $F$-regular domain if (and only if) so is $A$.
\end{theorem}
\begin{proof}
If $A$ is not a domain, then neither is $B$. Hence, both of them fail to be strongly $F$-regular as they are local and strongly $F$-regular algebras are products of strongly $F$-regular domains; see \cite{HochsterHunekeTightClosureAndStrongFRegularity}. Therefore $s(A)=0=s(B)$ if $A$ is not a domain. 

Assuming $A$ is a domain, set $q\mathrel{:} \Spec B \to \Spec A$ and 
\[
\delta\coloneqq \dim A + \log_p\left[\kappa(\aaa)^{1/p}:\kappa(\aaa)\right] = \dim B + \log_p\left[\kappa(\bbb)^{1/p}:\kappa(\bbb)\right].\] 
Then,
\begin{align*}
  \left[\kappa(\bbb):\kappa(\aaa)\right]  \cdot s(B)
  &= \left[\kappa(\bbb):\kappa(\aaa)\right] \cdot
    \lim_{e \rightarrow \infty}{\frac{1}{p^{e \delta }}
    \lambda_B \big(\Hom_B(F^e_* B, B)\bigl/\Hom_B(F^e_* B, \bbb)\big) }\\ 
  &= \lim_{e \rightarrow \infty}{\frac{1}{p^{e \delta }}
    \lambda_A \big(
    q_*\Hom_B(F^e_* B, B)\bigl/q_*\Hom_B(F^e_* B, \bbb)
    \big)}\eqqcolon  \mathfrak{C},
\end{align*}
where $\lambda_{-}$ is used to denote lengths. 

Our first hypothesis is $q^! A = \Hom_A(B,A) = B \cdot T$. Then, by Grothendieck duality for $q$, composing-with-$T$ yields an $A$-isomorphism $\tau\mathrel{:}q_*\Hom_B(F^e_* B, B) \to \Hom_A(q_* F^e_* B, A)$ under which $\tau( q_*\Hom_B(F^e_* B, \bbb) ) = \Hom_A(q_*F^e_* B, \aaa)$. Indeed, the inclusion ``$\subset$'' follows at once from $T(\bbb) \subset \aaa$. The converse containment; rather its contrapositive, follows from the surjectivity of $T$, for $\tau(\varphi)=T \circ \varphi$ is surjective if so is $\varphi$. Our analysis resumes as follows
\begin{align*}
  \mathfrak{C}
  &= \lim_{e \rightarrow \infty}{\frac{1}{p^{e \delta }}
    \lambda_A \big(\Hom_A(q_*F^e_* B, A)\bigl/\Hom_A(q_*F^e_* B, \aaa)\big)}\\ 
  &= \lim_{e \rightarrow \infty}{\frac{1}{p^{e \delta }}
    \lambda_A \big(\Hom_A(F^e_* q_* B, A)\bigl/\Hom_A(F^e_* q_*B, \aaa)\big)}\\
  &= \dim_{K(A)} {B_{K(A)}} \cdot s(A).
\end{align*}
See \cite[Theorem 4.11]{TuckerFSigExists} for the last equality (\cf \cite[Proposition 3.5, Lemma 3.6]{BlickleSchwedeTuckerFSigPairs1}).
\end{proof}

The following result is implicit in the proof of \autoref{thm.TransformationRule}.
\begin{scholium} \label{sch.Scholium2}
Work in the setup of \autoref{thm.TransformationRule}. $A$ is $F$-pure $($if and$\,)$ only if so is $B$.
\end{scholium}
\begin{proof}
The ``$\Leftarrow$'' direction is well-known; see \cite[Proposition
1.10]{SmithZhang}. Conversely, for every $\varphi \in \Hom_A(F^e_*
A,A)$ there exists a unique  $\psi \in \Hom_B(F^e_* B,B)$ such that $T
\circ \psi = \varphi \circ F^e_* T$.  Since $T(\bbb) \subset \aaa$ and
$F_*^eT$ is surjective, $\psi$ is surjective if so is $\varphi$. 
\end{proof}

In order to apply \autoref{thm.TransformationRule} to a finite
$G$-quasitorsor $(R, \mmm, \kay, K) \subset (S, \nnn, \kay, L)$, we must
check that $\Tr_{S/R}$ satisfies the three hypotheses in \autoref{thm.TransformationRule}. The first one follows from \autoref{cor.Tr_generates} and $S$ being $\mathbf{S}_2$ in \autoref{pro.TechnicalProposition}. The second one follows from $R$ being a splinter; see \autoref{pro:SplinterImpliesTame}. The third hypothesis, however, will occupy us for the rest of this section. In case it holds, $s(S)=o(G) \cdot s(R)$. It is satisfied in the \'etale case;
see \cite[Lemma 2.10]{CarvajalSchwedeTuckerEtaleFundFsignature}, \cite[Lemma 9]{SpeyerFrobeniusSplit}. The following example warns us to be careful in general.

\begin{example} \label{ex.ExampleBadTrace}
Let $S=R[t]/( t^p-r )$ and consider an arbitrary field $\kay$. Note that $S$ is local for all $r$. However, what its maximal ideal $\nnn$ is depends on $r$. Let $y$ be the closed point of $\Spec S$ and $x$ be the one of $\Spec R$. If $r \in \mmm$, then $\nnn = \mmm \oplus R\cdot t \oplus \cdots \oplus R\cdot t^{p-1}$ and so $y$ is a $\kay$-rational point lying over $x$. If $r \notin
\mmm$, there are two cases depending on whether or not $r$ has a $p$-th root residually. If $r=u^p+z$ for some $u \in R^{\times}$ and $z\in \mmm$, then $\nnn=\mmm S + (t-u)$, so $y$ is a $\kay$-rational point too.  If $r$ has no $p$-th roots even residually, then $\nnn = \mmm S$. However, that would be impossible if we demand $y$ to be a $\kay$-rational point as we have $\kay \subsetneq \kay(r^{1/p})$ at the residue fields level.

Now, $(R, \mmm)\subset (S,\nnn)$ is an $\alpha_p$-torsor for all $r \in R$ via the coaction $\alpha\mathrel{:}t \mapsto t \otimes 1 + 1 \otimes \xi$. If $r\in \mmm$,  $t^{p-1} \in \nnn$ as we saw above. However,
\begin{align*}
  \Tr_{S/R}(t^{p-1})
  &= \big(\id \otimes \Tr_{\alpha_p}\big) \big(\alpha(t^{p-1})\big)\\
&=\big(\id \otimes \Tr_{\alpha_p}\big) \big( (t\otimes 1 + 1\otimes \xi)^{p-1}\big)\\
&=\big(\id \otimes \Tr_{\alpha_p}\big) \Bigg( \sum_{i=0}^{p-1}{\binom{p-1}{i}t^{p-1-i}\otimes \xi^{i}}\Bigg)\\
&= \sum_{i=0}^{p-1}{\binom{p-1}{i}t^{p-1-i}\Tr_{\alpha_p} (\xi^i)}=1,
\end{align*}
see \autoref{ex.ExampleTraces}. Hence, $\Tr_{S/R}(\nnn) \not\subset \mmm$. This can also happen even for $\mu_p$-torsors. Indeed, if $r$ is a unit, then $(R, \mmm)\subset (S,\nnn)$ is a $\mu_p$-torsor under the coaction $t \mapsto t \otimes \zeta$. If $r=u^p+z$ as above, then $t-u \in \nnn$ but $\Tr_{S/R}(t-u)=u$ by a similar computation to the one before. Interestingly, if $r$ has no $p$-th roots even residually, $\Tr_{S/R}(\nnn) \subset \mmm$ and the transformation rule takes the form $p \cdot s(S) = p\cdot s(R)$, so $s(S)=s(R)$. In view of this, one may ask whether, if $r=u^p+z$ with $z \neq 0$, there is any chance that $s(S)\geq s(R)$. The following example discards this. Take $p=3$, $R=\kay\bigl\llbracket s,z^3 \bigr\rrbracket$, and $r=1+z^6$. Then, $S = \kay \bigl\llbracket s , z^2, z^3 \bigr\rrbracket$, which is not even normal. The extra variable ``$s$'' is for the sake of having a $2$-dimensional example.
\end{example}

In \autoref{ex.ExampleBadTrace}, $R\subset S$ was a torsor everywhere. This motivates the following question.

\begin{question} \label{que.Question} Work in \autoref{setup}. Let $(R, \mmm, \kay,K) \subset (S, \nnn, \kay, L)$ be a finite $G$-quotient that restricts to a $G$-torsor over $U$ but not everywhere. Is $\Tr_{S/R}(\nnn) $ contained in $\mmm$? 
\end{question}

We investigate \autoref{que.Question} for unipotent and linearly reductive group-schemes separately.

\subsection{The unipotent case} We show that \autoref{que.Question} is empty if $G$ is unipotent. 

\begin{theorem}\label{thm.UnipotentTorsors}
Work in \autoref{setup}. The restriction map $\varrho^1_{X,U}(G) \mathrel{:} \check{H}^1(X_{\textnormal{fl}}, G) \to \check{H}^1(U_{\textnormal{fl}}, G)$ is surjective for all unipotent finite algebraic groups $G/\kay$.
\end{theorem}

We will provide two proofs. The first one is an application of the work in \cite{ChinburgErezPappasTaylorTameActions} and hence shorter looking. The second proof was our original approach and is quite direct. We consider the techniques involved in our proof to be quite valuable and interesting in their own right. In fact, we reuse them in the proof of our main theorem (\autoref{thm.MainTheorem}).

\begin{proof}[First proof of \autoref{thm.UnipotentTorsors}]
According to \cite[Proposition 6.2]{ChinburgErezPappasTaylorTameActions}, if $R \subset S$ is a tame $G$-quotient by a unipotent finite algebraic group $G/\kay$, then $R \subset S$ must be a $G$-torsor. Therefore, the result follows from \autoref{pro.TechnicalProposition} and \autoref{pro:SplinterImpliesTame}. 
\end{proof}

The second proof relies on unipotent group-schemes admitting a central normal series whose intermediate quotients are (isomorphic to) subgroups of $\mathbb{G}_a$ and in particular commutative; see \cite[14.21]{MilneAlgebraicGroups}. In view of this, we will show \autoref{thm.UnipotentTorsors} first in the commutative case and obtain the general case by doing induction on the order.

\subsubsection{The commutative unipotent case} \label{ElementaryCase} If $G$ is commutative and $I=\mmm$ (\ie $U$ is the punctured spectrum),  we have the following exact sequence from \cite[III, Corollaire 4.9]{BoutotSchemaPicardLocal}:
\begin{equation} \label{eqn.BuototSES}
0 \to H^1(X_{\textnormal{fl}},G) \xrightarrow{\rho_{X,U}^1(G)} H^1(U_{\textnormal{fl}},G) \to \Hom(G^{\vee}, \textnormal{Pic}_{R/\kay}(U))\to 0.
\end{equation}
Recall that, since $G$ is commutative, the group $H^i(X_{\textnormal{fl}},G)$ (defined via derived-functor cohomology of sheaves of abelian groups) is isomorphic to the group $\check{H}^i(X_{\textnormal{fl}},G)$ (defined via torsors and gerbes); see \autoref{def.torsors}. Thus, every commutative $G$-torsor over $U$ extends to a $G$-torsor
over $X$ if and only if $\Hom(G^{\vee}, \textnormal{Pic}_{R/\kay}(U))$
is trivial. This could be used to simplify our forthcoming
arguments. However, to the best of the author's knowledge, it is
unknown whether Boutot's theory of the local Picard scheme and
\autoref{eqn.BuototSES} extend to general $I$ of height $\geq 2$. This
is rather limiting for us. First, we are interested in obtaining
potential global results like those in
\cite{BhattCarvajalRojasGrafSchwedeTucker}. Second, the case $I=\mmm$
is most interesting for surfaces singularities yet we are interested
in higher dimensions. We take a closer look at Boutot's arguments in
\cite[III]{BoutotSchemaPicardLocal} to see what information on the
cokernel of $\rho^1_{X,U}(G)$ we can still get. We introduce the
following notation
\begin{equation} \label{CokernelObstruction}
\Ob_{X,U}(G) \coloneqq \coker ( \rho^1_{X,U}(G)\mathrel{:}H^1(X_{\textnormal{fl}}, G) \to H^1(U_{\textnormal{fl}}, G) )
\end{equation}

\begin{lemma}\label{lem.KeyLemma}
Let $\ast \to G' \to G \to G'' \to \ast$ be a short exact sequence of commutative affine group-schemes over $\kay$ with $\Ob_{X,U}(G')=0=\Ob_{X,U}(G'')$. Then, $\Ob_{X,U}(G)=0$.
\end{lemma}
 \begin{proof}
Recall that $\rho_{X,U}^i(-)\mathrel{:}H^i(X_{\textrm{ft}},-) \to H ^i(U_{\textrm{ft}},-)$ is obtained as the left-derived natural transformation of $\Gamma(X,-)\to \Gamma(U,-)$ and so it is compatible with the $\delta$-structures. We then have the following commutative and horizontally-exact diagram:
\[
 \xymatrix{
H^0(X_{\textnormal{fl}},G'') \ar[r]^-{\delta} \ar[d] & H^1(X_{\textnormal{fl}},G') \ar[r] \ar[d] & H^1(X_{\textnormal{fl}},G) \ar[r] \ar[d]& H^1(X_{\textnormal{fl}},G'') \ar[r]^-{\delta} \ar[d] & H^2(X_{\textnormal{fl}},G') \ar[d] \\
H^0(U_{\textnormal{fl}},G'') \ar[r]^-{\delta}  & H^1(U_{\textnormal{fl}},G') \ar[r] & H^1(U_{\textnormal{fl}},G) \ar[r]& H^1(U_{\textnormal{fl}},G'') \ar[r]^-{\delta} & H^2(U_{\textnormal{fl}},G') 
}
\] 
Our hypothesis is that second and fourth vertical arrows are surjective. Hence, according to the $5$-lemma, to get surjectivity of the third one, we need the fifth arrow to be injective. However, Boutot shows in \cite[III, Corollaire 4.9]{BoutotSchemaPicardLocal} that $H^2(X_{\textnormal{fl}},G')=0$ for all commutative $G'$. In fact, all cohomologies higher than $1$ vanish; see \cite[Proposition 3.1]{BhattAnnihilatingCohomologyGroupSchemes}.
\end{proof} 
 
The following proposition demonstrates \autoref{thm.UnipotentTorsors} in the commutative case.

\begin{proposition} 
If $G$ is a commutative unipotent finite algebraic group then $\Ob_{X,U}(G)=0$.
\end{proposition}
\begin{proof} 
By \autoref{lem.KeyLemma}, it suffices to treat the simple case,  \textit{i.e.}  $G=\Z/p$ and $G=\alpha_p$.

\begin{claim} \label{cla.Z/pZTorsors} $\Ob_{X,U}(\Z/p)=0$.
\end{claim}
\begin{proof}[Proof of the claim]
This is a consequence of Artin--Schreier theory; see \cite[III, Proposition 4.12]{MilneEtaleCohomology}. From the long exact sequence on flat cohomology derived from \autoref{eqn.ArtinSchreierSES}, we have
\[
\Ob_{X,U}(\Z/p) \cong H^1(U, \sO_U)^F\coloneqq\bigl\{a\in H^1(U, \sO_U)=H^2_{I}(R) \mid Fa =a\bigr\}.\]
This vanishes if the ring $R$ is just $F$-rational and $I=\mmm$ by \cite[\S 2, Theorem 2.6]{SmithFRatImpliesRat}. One proves $H^1(U, \sO_U)^F=0$ in our setting as follows. Take $a\in H^1(U, \sO_U)^F$ and let $r$ be a nonzero element in the annihilator of $a$ (recall that every element of $H_I^i(R)$ is annihilated by some power of $I$). Let $\varphi \in \Hom_R(F^e_* R, R)$ be a splitting of the $R$-linear composition $R \to F^e_* R \xrightarrow{\cdot r} F^e_* R$. By applying $H_I^2(-)$, we get that $\phi\coloneqq H_I^2(\varphi)$ splits the composition $H_I^2(R) \xrightarrow{F^e} H_I^2(R) \xrightarrow{\cdot r} H_I^2(R)$, whence $a= \phi ( r \cdot F^e a ) = \phi(r \cdot a)=\phi (0) =0 $. 
\end{proof}
\begin{claim} \label{cla.AlphapTorsors}  $\Ob_{X,U}(\alpha_p)=0$.
\end{claim}
\begin{proof}[Proof of the claim] From the long exact sequence on flat cohomology derived from \autoref{eqn.InfinitesimalSES}, we get
\[
\Ob_{X,U}(\alpha_p)\cong \ker\left( H^1(U, \sO_U) \xrightarrow{F} H^1(U,\sO_U) \right).
\]
This kernel is zero, by definition, for an $F$-injective $X$ and $I=\mmm$. For general $I$, one can use $F$-purity to show $H_I^2 (R) \xrightarrow{F} H_I^2 (R) $ is injective. Indeed, if $\varphi$ splits $R \to F^e_*R$, then $H_I^2 (\varphi)$ splits $H_I^2 (R) \xrightarrow{F} H_I^2 (R) $, forcing it to be injective. This proves the claim.
\end{proof}
The proposition then holds.
\end{proof}

\begin{remark}
\autoref{cla.Z/pZTorsors} follows for $R$ a splinter from \cite[Corollary 2.11]{CarvajalSchwedeTuckerEtaleFundFsignature}.
\end{remark}
\subsubsection*{Acknowledgement}{The author would like to thank Christian Liedtke who first observed the results of \autoref{cla.Z/pZTorsors} and \autoref{cla.AlphapTorsors}} and made him aware of them. These results will be treated in an upcoming work by Christian Liedtke and Gebhard Martin \cite{LiedtkeMartin}.

\subsubsection{The general unipotent case} \label{GeneralUnipotentCase} 
To handle the general case, we proceed by induction on the order of $G$ along with the fact it admits a central commutative (necessarily) unipotent subgroup whose quotient is (necessarily) unipotent \cite[14.21]{MilneAlgebraicGroups}. However, we shall require the use of nonabelian first and second flat cohomology as treated in \cite{GiraudNonabelianCohomology,DebremaekerNonAbelianCohomology}.

\begin{proof}[Second proof of \autoref{thm.UnipotentTorsors}]
Let $G'$ be a nontrivial central (so normal and commutative) subgroup
of $G$ with corresponding short exact sequence 
$*\to G' \to G \to G'' \to *$ so that $G''$ is unipotent and 
$o(G'')< o(G)$. We may assume $o(G')<o(G)$ as otherwise we are done by
\autoref{ElementaryCase}. Consider now the commutative diagram
\begin{equation} \label{eqn.CommutativeDiagram}
  \xymatrix{
    H^0(X_{\textnormal{fl}},G'') \ar[r]^-{} &
    \check{H}^1(X_{\textnormal{fl}},G') \ar[r] \ar[d] &
    \check{H}^1(X_{\textnormal{fl}},G) \ar[r] \ar[d]&
    \check{H}^1(X_{\textnormal{fl}},G'') \ar[r]^-{} \ar[d] &
    \check{H}^2(X_{\textnormal{fl}},G') \\
    H^0(U_{\textnormal{fl}},G'') \ar[r]^-{}  &
    \check{H}^1(U_{\textnormal{fl}},G') \ar[r] &
    \check{H}^1(U_{\textnormal{fl}},G) \ar[r]&
    \check{H}^1(U_{\textnormal{fl}},G'') \ar[r]^-{} &
    \check{H}^2(U_{\textnormal{fl}},G')  
  }
\end{equation} 
where the horizontal sequences are exact sequences of pointed sets. As
before,
$\check{H}^2(X_{\textnormal{fl}},G')=H^2(X_{\textnormal{fl}},G')=0$,
for $G'$ is commutative \cite[III, Corollaire
4.9]{BoutotSchemaPicardLocal}.  The first and third vertical arrows
are surjective by the inductive hypothesis. Unfortunately, we cannot
apply the $5$-lemma to get the surjectivity of the middle one as the
sets in consideration are no longer groups. Let us chase the diagram
to see how to get around this. Let 
$t_0 \in \check{H}^1(U_{\textnormal{fl}},G)$. It maps to 
$t_1 \in \check{H}^1(U_{\textnormal{fl}},G'')$ which extends to 
$t_2 \in \check{H}^1(X_{\textnormal{fl}},G'')$. However, $t_2$ lifts
to $t_3 \in \check{H}^1(X_{\textnormal{fl}},G)$ as
$\check{H}^2(X_{\textnormal{fl}},G')$ is trivial. Let 
$t_4 \in \check{H}^1(U_{\textnormal{fl}},G)$ be the restriction of
$t_3$ to $U$. At this point, we wish to substract $t_4$ from $t_0$ as
$t_4 \mapsto t_1$. However, such operation does not make sense in this
setting. Fortunately, we can make sense of this by \emph{changing the
origin via twisted forms}; see \cite[III,
\S2.6]{GiraudNonabelianCohomology}. Indeed, consider the conjugate
representation $G \to \Aut G$ of $G$ being defined by the action of
$G$ on itself by inner automorphisms. This gives a map of pointed sets
$\check{H}^1(U_{\textnormal{fl}}, G) \to
\check{H}^1(U_{\textnormal{fl}}, \Aut G)$, 
where $\check{H}^1(U_{\textnormal{fl}}, \Aut G)$ classifies the
so-called \emph{twisted forms} of $G$ \cite[III,
\S2.5]{GiraudNonabelianCohomology}. Let us write $t \mapsto {^t}G$ for
the map realizing $G$-torsors as twisted forms of $G$. Moreover, we
have a bijection: 
$\theta_t: \check{H}^1(U_{\textnormal{fl}}, G) \to
\check{H}^1(U_{\textnormal{fl}}, {^t}G)$ 
where $t$ is sent to the trivial class in
$\check{H}^1(U_{\textnormal{fl}}, {^t}G)$. If $G$ is commutative,
$\theta_t \mathrel{:} t' \mapsto t'-t$; see \cite[III,
2.6.3]{GiraudNonabelianCohomology}.

Thus, what we need to do is to twist $* \to G' \to G \to G'' \to *$ by
the $G$-torsor $t_4$. Indeed, $G$ acts by inner automorphism on both
$G'$ and $G''$. Then, $t_4$ can be realized as a twisted form of both
$G'$ and $G''$. Since $G'$ is central, the corresponding twisted form
is the trivial one, namely $G'$ itself. Since $t_4 \mapsto t_1$, the
twisted form of $G''$ given by $t_4$ is ${^{t_1}}G''$. Thus, we have a
short exact sequence $* \to G' \to {^{t_4}}G \to {^{t_1}}G'' \to *$
(on $U$) and a commutative diagram
\begin{equation} \label{eqn.comm}
\xymatrix{
H^1(U_{\textnormal{fl}}, G') \ar[r] &\check{H}^1(U_{\textnormal{fl}}, G) \ar[r] \ar[d]^-{\theta_{t_4}} &\check{H}^1(U_{\textnormal{fl}}, G'')\ar[d]^-{\theta_{t_1}}\\
H^1(U_{\textnormal{fl}}, G') \ar[r] &\check{H}^1(U_{\textnormal{fl}}, {^{t_4}}G) \ar[r] &\check{H}^1(U_{\textnormal{fl}}, {^{t_1}}G'')
}
\end{equation}
See \cite[III, 3.3.5]{GiraudNonabelianCohomology}. 
Likewise, we twist $* \to G' \to G \to G'' \to *$ by $t_3$ to get the
short exact sequence $* \to G' \to {^{t_3}}G \to {^{t_2}}G'' \to *$ on $X$. Further, we have the following commutative diagram 
\[
\xymatrix{
H^1(X_{\textnormal{fl}}, G') \ar[r] \ar[d] &\check{H}^1(X_{\textnormal{fl}}, {^{t_3}}G) \ar[r] \ar[d] &\check{H}^1(X_{\textnormal{fl}}, {^{t_2}}G'')\ar[d]\\
H^1(U_{\textnormal{fl}}, G') \ar[r] &\check{H}^1(U_{\textnormal{fl}}, {^{t_4}}G) \ar[r] &\check{H}^1(U_{\textnormal{fl}}, {^{t_1}}G'')
}
\]
with exact horizontal sequences. By \autoref{eqn.comm},
$\theta_{t_4}(t_0)\in \check{H}^1(U_{\textnormal{fl}}, {^{t_4}}G)$
maps to $\theta_{t_1}(t_1) \in
\check{H}^1(U_{\textnormal{fl}}, {^{t_1}}G'')$, which is the trivial
${^{t_1}}G$-torsor. Hence, there exists $t_5 \in
H^1(U_{\textnormal{fl}}, G') $ mapping to
$\theta_{t_4}(t_0)$. Nevertheless, by the inductive hypothesis, $t_5$
extends to a torsor $t_6 \in H^1(X_{\textnormal{fl}}, G')$, which maps
to a torsor $t_7 \in \check{H}^1(X_{\textnormal{fl}}, {^{t_3}}G) $. By
commutativity, $t_7$ restricts to $\theta_{t_4}(t_0)$. On the other
hand, we have the commutative square 
\[
\xymatrix{
\check{H}^1(X_{\textnormal{fl}}, G) \ar[d]  \ar[r]^-{\theta_{t_3}} &\check{H}^1(X_{\textnormal{fl}}, {^{t_3}}G)\ar[d]\\
\check{H}^1(U_{\textnormal{fl}}, G)  \ar[r]^-{\theta_{t_4}} &\check{H}^1(U_{\textnormal{fl}}, {^{t_4}}G)
}
\]
from which it follows that the unique $t_8 \in \check{H}^1(X_{\textnormal{fl}}, G)$ such that $\theta_{t_3}(t_8) = t_7$ restricts to $t_0$, \ie $\varrho^1_X(G)(t_8)=t_0$; as desired.
\end{proof}

\subsubsection*{Acknowledgement} The author is deeply thankful to Bhargav Bhatt who kindly taught him the use of twisted forms to control the fibers of a map $\check{H}^1(Y_{\textnormal{fl}},G) \to \check{H}^1(Y_{\textnormal{fl}},G/H)$.

From the above,  we extract the following useful lemma improving upon \autoref{lem.KeyLemma}.

\begin{lemma} \label{lemmaa.InductiveStep}
Work in \autoref{setup}.  Let $\ast \to G' \to G \to G'' \to \ast$ be an exact sequence of affine $\kay$-group-schemes such that $\varrho_{X,U}(G')$ and $\varrho_{X,U}(G'')$ are surjective.  If $\check{H}^1(X_{\mathrm{fl}},G) \to \check{H}^1(X_{\mathrm{fl}},G'')$ is surjective (\textit{e.g.}  $G'$ is solvable or $G= G'\rtimes G''$), then so is $\varrho_{X,U}(G)$.
\end{lemma}
\begin{proof}
We start by mentioning why $\check{H}^1(X_{\mathrm{fl}},G) \to \check{H}^1(X_{\mathrm{fl}},G'')$ is surjective if either $G'$ is solvable or $G= G'\rtimes G''$.  Consider the diagram \autoref{eqn.CommutativeDiagram}.  Note that $\check{H}^2(X_{\textnormal{fl}},G')$ is trivial if $G'$ is solvable.  Indeed,  this follows from the commutative case (\cite[Proposition 3.1]{BhattAnnihilatingCohomologyGroupSchemes}) by induction on the order and taking long exact sequences on nonabelian cohomologies. Finally, if $G\to G''$ is a split epimorphism then so is $\check{H}^1(X_{\mathrm{fl}}, G ) \to \check{H}^1(X_{\mathrm{fl}}, G'' )$; so it is surjective. 

We chase \autoref{eqn.CommutativeDiagram} and define $t_0$,  $t_1$,  $t_2$,  $t_3$, and $t_4$ as above.  Next, we twist the whole diagram by $t_3$ to obtain a commutative diagram
\[
\xymatrix{
\check{H}^1(X_{\textnormal{fl}}, {^{t_3}}G') \ar[r] \ar[d] &\check{H}^1(X_{\textnormal{fl}}, {^{t_3}}G) \ar[r] \ar[d] &\check{H}^1(X_{\textnormal{fl}}, {^{t_2}}G'')\ar[d]\\
\check{H}^1(U_{\textnormal{fl}}, {^{t_4}}G') \ar[r] &\check{H}^1(U_{\textnormal{fl}}, {^{t_4}}G) \ar[r] &\check{H}^1(U_{\textnormal{fl}}, {^{t_1}}G'')
}
\]
where $\theta_{t_4}(t_0) \in  \check{H}^1(U_{\textnormal{fl}}, {^{t_4}}G)$ is mapped to the trivial class in $\check{H}^1(U_{\textnormal{fl}}, {^{t_1}}G'')$. Therefore, $\theta_{t_4}(t_0)$ has a preimage $t_5 \in H^1(U_{\textnormal{fl}}, {^{t_4}}G')$. On the other hand, we consider the commutaive diagram
\[
\xymatrix{
\check{H}^1(X_{\textnormal{fl}}, G') \ar[d]  \ar[r]^-{\theta_{t_3}} &\check{H}^1(X_{\textnormal{fl}}, {^{t_3}}G')\ar[d]\\
\check{H}^1(U_{\textnormal{fl}}, G')  \ar[r]^-{\theta_{t_4}} &\check{H}^1(U_{\textnormal{fl}}, {^{t_4}}G')
}
\]
where the horizontal arrows are bijections. Since the left arrow is surjective, then so is the right arrow. Thus, we may consider $t_6 \in \check{H}^1(X_{\textnormal{fl}}, {^{t_3}}G') $ to be a preimage of $t_5$. Let $t_7 \in \check{H}^1(X_{\textnormal{fl}}, {^{t_3}}G)$ be the image of $t_6$. Notice that $t_7$ maps to $\theta_{t_4}(t_0)$ in $\check{H}^1(U_{\textnormal{fl}}, {^{t_4}}G)$. Therefore, the untiwst of $t_7$, say $t_8 \coloneqq \theta_{t_3}^{-1}(t_7) \in \check{H}^1(X_{\textnormal{fl}}, G)$, maps to $t_0$ under $\check{H}^1(X_{\textnormal{fl}}, G) \to \check{H}^1(U_{\textnormal{fl}},G)$. 
\end{proof}

\subsection{The linearly reductive case} We keep working in \autoref{setup}. However, most of our discussion in this section can be relaxed to $R$ not being necessarily strictly henselian. Recall that linearly reductive finite algebraic groups are extensions of \'etale groups with prime-to-$p$ order by $D(\Gamma)$ where $\Gamma$ is an abelian $p$-group. Hence, we may focus on $\mu_{p^e}$. By Kummer Theory \cite[Chapter III, \S4]{MilneEtaleCohomology}, $\mu_n$-torsors over $U$ are (up to isomorphisms) in one-to-one correspondence with pairs $(\sL, \varphi)$, where $\sL$ is an invertible sheaf on $U$ and $\varphi$ is an isomorphism $\sO_U \to \sL^{n}$. The $\mu_n$-torsor, say $V \to U$, associated to the pair $(\sL, \varphi)$ is the cyclic cover $\mathbf{Spec}_U \bigoplus_{i=0}^{n-1} \sL^{i}$. That is, $V$ is the open subscheme of $\Spec C(\sL,\varphi)$ lying over $U$ where $C(\sL,\varphi)$ is the following finite $R$-algebra; see \autoref{pro.TechnicalProposition}. As an $R$-module, $C(\sL, \varphi)$ equals $\bigoplus_{i=0}^{n-1}H^{0}(U, \sL^i)$. Its ring product is given by the $R$-linear maps:
\begin{align*}
&H^0(U, \sL^i) \otimes_R H^0(U, \sL^j) \xrightarrow{\textnormal{can.}} H^0(U, \sL^{i+j}) &\text{ if } i+j < n,\\
&H^0(U, \sL^i) \otimes_R H^0(U, \sL^j) \xrightarrow{\textnormal{can.}}  H^0(U, \sL^{i+j}) \xrightarrow{H^0(U,\phi^{-1})} H^0(U, \sL^{i+j-n}) &\text{ if } i+j\geq n.
\end{align*}
The coaction 
$C(\sL,\varphi) \to C(\sL,\varphi) \otimes \sO(\mu_n)$ is defined
by sending $f \in H^0(U,\sL^i)$ to 
$f \otimes \zeta^i$. \emph{Therefore, the trace
$\Tr_{C(\sL,\varphi)/R}$ is the projection onto the zeroth-degree
summand in the above direct sum.}  See \autoref{ex.ExampleTraces}. Let
us write $\sL = \sO_U(D)$ for some Cartier divisor $D$ on $U$. Since
$U$ is big, $D$ extends uniquely to a Weil divisor on $X$. Thus,
\[
  H^0(U,\sL^i)=H^0(U,\sO_U(iD))=R(iD)=\{f \in K^{\times} \mid
  \divisor (f) + iD\geq 0\} \cup \{0\} \subset K
.\]
Defining an isomorphism $\varphi \mathrel{:} \sO_U \to \sO_U(nD)$ amounts to give
$a\in K^{\times}$ such that $\divisor(a)+nD =0$, which implies $R(nD)
= R \cdot a \subset  K$. Hence, we may present the data of a cyclic
cover as $C(\sL,\varphi)=C(D,a;n)$. In this way, the ring product is
performed internally by the pairing 
\[\left\{\begin{array}{rcl}
R(iD)\otimes_R R(jD) & \longrightarrow & R((i+j)D)\\
f \otimes g & \longmapsto &  f\cdot g,
\end{array}\right.
\]
if $i+j<n$. In case $m\coloneqq i+j -n \geq 0$, we have:
\[
R(iD)\otimes_R R(jD)  \to  R((i+j)D) = R(mD+nD)   \xleftarrow{\cong}  R(mD) \otimes_R R(nD)  =  R(mD) \otimes_R R a  \xleftarrow{\cong}  R(mD), 
\]
where the first isomorphism ``$\xleftarrow{\cong}$'' uses that $nD$ is Cartier. Succinctly, if $i+j \geq n$, we have
\[\left\{\begin{array}{rcl}
R(iD)\otimes_R R(jD) & \longrightarrow & R((i+j-n)D) \\
f \otimes g & \longmapsto & fg \left/ a\right. .
\end{array}\right.
\]
See \cite{TomariWatanabeNormalZrGradedRings} for further details. Observe that, in Tomari--Watanabe's terms, we are considering only \emph{integral} cyclic covers. For convenience, we name condition \cite[1.7.1]{TomariWatanabeNormalZrGradedRings} as follows.

\begin{terminology} \label{ter.VeroneseKummerTypes}
$R \subset C(D,a;n)$ is a $\mu_n$-torsor if and only if $D$ is Cartier on $X$ (\textit{i.e.} trivial). In such a case, we say that the cyclic cover is of \emph{Kummer-type}. If $n$ is the index of $D$ (\textit{i.e.} condition \cite[1.7.1]{TomariWatanabeNormalZrGradedRings} holds), we say that it is of \emph{Veronese-type} in which case it is a $\mu_n$-quasitorsor. 
\end{terminology}

\begin{proposition} \label{pro.GoodTrace}
With notation as above, suppose that $R \subset C(D,a;n)$ is of Veronese-type. Then, $C=C(D,a;n)$ is a local domain with maximal ideal $\nnn_C = \mmm \oplus \bigoplus_{i=1}^{n-1}{R(iD)}$ and field of fractions $K(a^{1/n})$. In particular, \autoref{que.Question} has an affirmative answer in this case. Therefore, $\big(C,\nnn_C, \kay,K(a^{1/n})\big)$ is a strongly $F$-regular $\kay$-germ with $s(C)=n\cdot s(R)$ (or just $F$-pure if $R$ were assumed only $F$-pure).
\end{proposition}
\begin{proof}
For the first statement, see \cite[Corollary 1.9 and Lemma 2.1]{TomariWatanabeNormalZrGradedRings}. The rest follows from \autoref{sch.Scholium2}, \autoref{thm.TransformationRule}, and $C$ is strictly local by \cite[I, Corollary 4.3]{MilneEtaleCohomology}. As mentioned above, $\Tr_{C(\sL,\varphi)/R}$ is the projection onto the zeroth-degree summand in $S = \bigoplus_{i=0}^{n-1}{R(iD)}$. Since $\nnn_C = \mmm \oplus \bigoplus_{i=1}^{n-1}{R(iD)}$, it follows that $\Tr_{C(\sL,\varphi)/R}(\nnn_C) = \mmm$.
\end{proof}

\begin{remark} \label{rem.NotStrictlyHenselian}
\autoref{pro.GoodTrace} works just fine if we do not assume that $R$ is strictly henselian. In that case, $C$ would not be necessarily strictly henselian either.
\end{remark}

\begin{example} \label{ex.CanonicalCovers}
Suppose that $R$ is $\Q$-Gorenstein with canonical divisor $K_R$ of index $n$; say $\divisor a + nK_R=0$. The corresponding Veronese-type cyclic cover is called a \emph{canonical (or index-$1$) cover} of $R$. As a direct application of \autoref{pro.GoodTrace}, both strong $F$-regularity and $F$-purity transfer to canonical covers. Assuming $p \nmid n$, this was originally demonstrated by K.-i.~Watanabe \cite{WatanabeFRegularFPure} but remained open otherwise. It is worth noting that neither $F$-rationality nor KLT-ness are transferred to cyclic covers in general \cite{SinghCyclicCoversOfRational,KawamataIndex1CoversOfLogTerminalSurfaceSingularities}. 
\end{example}

Let us look at \autoref{que.Question} for cyclic covers, \textit{i.e.} $S=C$, $G = \mu_n$. We learned from \autoref{ex.ExampleBadTrace} that Kummer-type cyclic covers may be a source of problems for us whereas from \autoref{pro.GoodTrace} that Veronese-type ones are suitable for our purposes. A general cyclic cover is an intricate mixture of both types and so we cannot expect to have an affirmative answer for \autoref{que.Question} in general. At least, we have the following way out.

\begin{proposition} \label{lem.DecompositionCyclicCovers}
Work in the setup of \autoref{que.Question} assuming $G=\mu_n$. Then, there exist a nontrivial Veronese-type cyclic cover $(R,\mmm,\kay,K)\subset (S',\mathfrak{n}', \kay, L')$. In particular, $(S',\mathfrak{n}', \kay, L')$ is a strongly $F$-regular $\kay$-germ with $F$-signature at least $2 \cdot s(R)$. 
\end{proposition}
\begin{proof}
Write $S\cong C(D,a;n)$ as above. By hypothesis, the index of $D$ is not $1$. Thus, we may take $S'$ to be $C(D,b,m)$, where $m\neq 1$ is the index of $D$ and we choose $b \in K^{\times}$ so that $\divisor b + m D = 0$. For the last statement, use \autoref{pro.GoodTrace}.
\end{proof}

\subsection{Main results} We are ready to summarize with our main results. We commence by answering to \autoref{que.Question} when $G^{\circ}$ is trigonalizable.

\begin{theorem} \label{thm.MainTheorem}
Work in \autoref{setup}. Suppose that $\varrho_{X,U}^1(G)$ is not surjective and $G^{\circ}$ is trigonalizable. Then,  there exist a nontrivial finite linearly reductive group-scheme $G'/\kay$ and a $G'$-quotient $(R,\mmm,\kay,K) \subset (S',\nnn',\kay, L')$ restricting to a $G'$-torsor over $U$ but not everywhere such that the answer to \autoref{que.Question} is affirmative in this case.  In particular, $(S',\nnn',\kay,L')$ is a strongly $F$-regular $\kay$-germ with $s(S') = o(G')\cdot s(R)$ and so $o(G')\leq 1/s(R)$.
\end{theorem}
\begin{proof}
Using $G=G^{\circ} \rtimes \pi_0(G)$ as in \autoref{AlternateReduction} and \cite[Corollary 2.11]{CarvajalSchwedeTuckerEtaleFundFsignature}, we may assume that $G/\kay$ is infinitesimal  and further trigonalizable by hypothesis.  Thus, we may write $G = G_{\mathrm{u}} \rtimes G_{\mathrm{mt}}$ as in \autoref{sec.TrigonalizableGroups}.  By \autoref{lemmaa.InductiveStep} and \autoref{thm.UnipotentTorsors}, we may assume $G = G_{\mathrm{mt}} \cong \bigoplus_i \mu_{p^{e_i}}$.  Using \autoref{lem.KeyLemma},  we may say $G = \mu_{p^e}$ and the result then follows directly from \autoref{lem.DecompositionCyclicCovers}.
\end{proof}

\begin{theorem}[The existence of a maximal cover] \label{thm.FundamentalGroupIsFinite}
Work in \autoref{setup}. There exists a chain of finite extensions of strongly $F$-regular $\kay$-rational germs
\[
(R, \mmm, \kay,K) \subsetneq (R_1, \mmm_1, \kay,K_1) \subsetneq \cdots \subsetneq (R_t, \mmm_t, \kay, K_t) = (R^{\star}, \mmm^{\star}, \kay, K^{\star}),
\]
where each extension $(R_i, \mmm_i, \kay,K_i) \subsetneq (R_{i+1}, \mmm_{i+1}, \kay, K_{i+1})$ is a $G_i$-quasitorsor and:
\begin{enumerate}
\item $G_i/\kay$ is a linearly reductive finite algebraic group for all $i$, 
\item $[K^{\star}:K]$ is at most $1/s(R)$,
\item $\varrho^1_{X^{\star},U^{\star}}(G) \mathrel{:} \check{H}^1(X_{\mathrm{fl}}^{\star},G) \to \check{H}^1(U^{\star}_{\mathrm{fl}}, G)$ is surjective for all big opens $U^{\star} \subset X^{\star}$ and all finite algebraic groups $G/\kay$ for which $G^{\circ}$ is solvable.
\end{enumerate}
\end{theorem}
\begin{proof}
By iterating \autoref{thm.MainTheorem} until $s(R)$ is exhausted, we obtain the weaker version of this result where in part (c) we assume that $G^{\circ}$ is trigonalizable. To see part (b), just notice that the transformation rule in \autoref{thm.TransformationRule} yields
\begin{multline*}
1 \geq s(R^{\star}) = [K^{\star}:K_{t-1}] s(R_{t-1}) = [K^{\star}:K_{t-1}][K_{t-1}:K_{t-2}]s(R_{t-2}) \\
= \cdots = [K^{\star}:K_{t-1}] [K_{t-1}:K_{t-2}] \cdots [K_1: K] s(R) =[K^{\star}:K] s(R).
\end{multline*}
It remains to show part (c) when we let $G^{\circ}$ to be solvable.  As in the proof of \autoref{thm.MainTheorem},  it suffices to prove that $\varrho^1_{X^{\star},U^{\star}}(G)$ is surjective for all infinitesimal solvable groups $G/\kay$.  We proceed by induction on $o(G)$.  The base case is covered by the commutative case.  We may assume $G$ is not commutative and so that there is a nontrivial proper normal subgroup $G' \subset G$; which is necessarily solvable,  such that $G''=G/G'$ is commutative.  The result follows by applying \autoref{lemmaa.InductiveStep} and the inductive hypothesis.
\end{proof}

\begin{remark} \label{rem.EtaleInfinitesimalDecomposition}
The maximal cover $R \subset R^{\star}$ in \autoref{thm.MainTheorem} can be constructed as $R \subset R^{\mathrm{\Acute{e}t}} \subset R^{\star}$ where $R \subset R^{\mathrm{\Acute{e}t}}$ is the universal Galois quasi-\'etale cover from \cite{CarvajalSchwedeTuckerEtaleFundFsignature}, which is tame of prime-to-$p$ degree, and $R^{\mathrm{\Acute{e}t}} \subset R^{\star}$ is obtained as a finite chain of infinitesimal Veronese-type cyclic covers and so it is purely inseparable of degree $p^e$ for some $e \geq 0$.  \emph{It is unclear to the author whether $R \subset R^{\star}$ can be realized as a $G^{\star}$-quasitorsor for some linearly reductive algebraic group $G^{\star}/\kay$. } It is even unclear to him whether this can be done for $R^{\mathrm{\acute{e}t}} \subset R^{\star}$ (see \cite[Example 2.7]{TomariWatanabeNormalZrGradedRings}).  However,  this is true in dimension $2$ by \cite[Theorem 11.2]{LiedtkeMartinMatsumotoLinearlyreductiveQuotientSingularities}.
\end{remark}

\subsection{Relation with Esnault--Viehweg's 
  local Nori fundamental group-scheme}
It is natural to wonder what is the relationship between our maximal
cover $R \subset R^{\star}$; see
\autoref{thm.FundamentalGroupIsFinite}, and Esnault--Viehweg's
construction of the \emph{local Nori fundamental group-scheme} in
\cite{EsnaultViehwegSurfaceSingularitiesDominatedSmoothVarieties}. For
simplicity, let us consider the commutative case only. Then, the
property our cover has is that its abelian local Nori fundamental
group-scheme
$\pi_{1,\textnormal{loc}}^{\textnormal{N,ab}}(U^{\star},X^{\star},x^{\star})$
is trivial. Nevertheless, it is not clear to the author whether or not
this implies that $\pi_{1,\textnormal{loc}}^{\textnormal{N,ab}}(U,X,x)
$ is finite. He is deeply thankful to Christian Liedtke who pointed
out to him an example suggesting this should not be always the
case. The example is a surface $\textnormal{D}_4$ singularity in
characteristic $2$ that is a $\Z/2\Z$-quotient of
$\hat{\mathbb{A}}_{\kay}^2$ but admits a nontrivial
$\mu_2$-quasitorsor. The singularity is $\kay\llbracket
x,y,x\rrbracket/(z^2-xyz-x^2y-xy^2 )$. The divisor defined by $(x,z)$
has index $2$. Therefore, its local abelian Nori fundamental
group-scheme cannot be $\Z/2\Z$ as it must take into account
$\mu_2$ as well. This singularity  (although $F$-pure) is not
strongly $F$-regular. See \cite[Example
7.12]{SchwedeTuckerTestIdealFiniteMaps} for a closer look into this
singularity, \cf \cite{ArtinWildlyRamifiedZ2Actions}.

\subsection{Beyond the solvable case}
The author is grateful to Axel St\"{a}bler for making him aware of the
classification of simple infinitesimal rank-$1$ algebraic groups via
the classification of simple restricted Lie algebras by
Block--Wilson--Premet--Strade
\cite{BlockWilsonClassification,StradeWilsonClassification,PremetStradeCompletion}.
See \cite{VivianiSimpleFIniteGroupSchemes} for a nice account. With
$\varrho^1_{X,U}(G)$ as above, the following questions are of
interest.

\begin{question}
For which simple infinitesimal rank 1 algebraic groups $G/\kay$ is $\varrho^1_{X,U}(G)$ surjective? Is this the case for those $G/\kay$ of Cartan type?
\end{question}
\begin{question}
If $G/\kay$ is a simple infinitesimal rank 1 algebraic group such
that $\varrho^1_{X,U}(G)$ is not necessarily surjective, do we have an
analog of \autoref{pro.GoodTrace}?
\end{question}
\begin{question}
Given a simple infinitesimal rank 1 algebraic group $G/\kay$, for
which type of ($F$-)singularity $X$, if any, is
$\varrho^1_{X,X\smallsetminus \{\mmm\} }(G)$ naturally surjective? 
\end{question}
\begin{question}
Let $S$ be the spectrum of a strictly local ring and $G/S$ be a (not
necessarily solvable) finite flat group-scheme. Does the vanishing
$\check{H}^2(S_{\textnormal{fl}}, G)=0$ hold? 
\end{question}

As we have seen in the proofs of \autoref{thm.UnipotentTorsors} and
\autoref{thm.FundamentalGroupIsFinite}, we need to answer this last
question in order to bootstrap the simple case to the general case.

\section{Applications to divisor class groups} \label{sec:applications}

The methods in \cite{CarvajalSchwedeTuckerEtaleFundFsignature} were
able to bound the prime-to-$p$ torsion of $\Cl R = \Pic U$ from above
by $1/s(R)$. Using \autoref{pro.GoodTrace}, we are able to bound all
torsion. Recall that 
$\Cl R \subset \Cl R^{\mathrm{h}} \subset \Cl \hat{R}$ for a normal
local domain $R$; see \cite[\S1C]{MumfordAlgebraicGeometryI}.

\begin{corollary} \label{cor.TorsionPicard}
Let $R$ be a local strongly $F$-regular ring. If $D \in \Cl R$ is torsion with index $n$, then $n\leq 1/s(R)$ $($\textit{i.e.} $\Cl R$ is $n$-torsion with $n \leq \lfloor 1/s(R) \rfloor)$. Thus, $\Cl R$ is torsion-free if $s(R)>1/2$.
\end{corollary}
\begin{proof}
Setting $\divisor a^{-1} = nD$, we get $n \cdot s(R)= s(C(D,a;n)) \leq
1$ by \autoref{pro.GoodTrace}, \autoref{rem.NotStrictlyHenselian}. 
\end{proof}

\begin{remark} 
In view of \autoref{cor.TorsionPicard}, one may wonder whether 
$(\Cl R)_{\mathrm{tor}}$ is finite with order at most $1/s(R)$. After
a former version of this preprint was public, this was addressed
effectively by T.~Polstra and I.~Martin
\cite{PolstraATheoremAboutMCMM,MArtinBoundOnClRforSFR}. We would like
to remark next that finiteness of $(\Cl R)_{\mathrm{tor}}$ follows
from our methods as well. However, the author acknowledges that he
only thought of this possibility after seeing Polstra's and Martin's
work. The argument is as follows. Let $R \subset R^{\star}$ be as in
\autoref{thm.MainTheorem}. Note that $\Cl R^{\star}$ is
torsion-free. Moreover, the kernel of the induced homomorphism 
$\Cl R \to \Cl R^{\star}$ is $(\Cl R)_{\mathrm{tor}}$. Indeed, 
$R \subset R^{\star}$ restricts to a finite faithfully flat cover
$U^{\star} \to U$ of degree $[K^{\star}:K]$ (where 
$U^{\star} \subset \Spec R^{\star}$ is a big open). In particular, 
$H \coloneqq \ker(\Pic U \to \Pic U^{\star})$ is 
$[K^{\star} : K]$-torsion; see \cite[Lemma
2.1]{GuralnickJaffeRaskindWiegandOnThePicardGroup}. Conversely, it is
clear that torsion elements of $\Cl R$ must be in $H$ as 
$\Pic U^{\star} = \Cl R^{\star}$ is torsion-free. It remains to
explain why $H$ is finite. We exploit that it is 
$[K^{\star} : K]$-torsion. Let us consider 
$R \subset R^{\mathrm{\Acute{e}t}} \subset R^{\star}$ as in
\autoref{rem.EtaleInfinitesimalDecomposition}, which induces 
$\Pic U \to \Pic U^{\mathrm{\acute{e}t}} \to \Pic U^{\star}$. It
suffices to prove that the kernel of both 
$\Pic U \to \Pic U^{\mathrm{\acute{e}t}}$ and 
$\Pic U^{\mathrm{\acute{e}t}} \to \Pic U^{\star}$ are finite. By an
argument as above, 
$(\Pic U)_{\mathrm{tor}}^{(p)} = \ker (\Pic U \to \Pic
U^{\mathrm{\acute{e}t}})$ 
is the prime-to-$p$ torsion subgroup (this uses that
$U^{\mathrm{\acute{e}t}} \to U$ is tame and 
$(\Pic U^{\mathrm{\acute{e}t}})_{\mathrm{tor}}^{(p)}=0$). However,
there exist a prime-to-$p$ integer $n$ such that 
$(\Pic U)_{\mathrm{tor}}^{(p)}$ is contained in the image of
$H^1(U_{\textnormal{fl}}, \mu_n) \to \Pic U$. Nonetheless,
\[
  H^1(U_{\textnormal{fl}}, \mu_n) \cong
  H^1(U_{\mathrm{\Acute{e}t}}, \Z/n\Z) \cong
  \Hom\big(\pi_1^{\mathrm{\Acute{e}t}}(U), \Z/n\Z\big) 
\]
is finite. To see why $\ker(\Pic U^{\mathrm{\acute{e}t}} \to \Pic
U^{\star})$ is finite, it suffices to consider the case where
$R^{\mathrm{\acute{e}t}} \subset R^{\star}$ is a single cyclic
cover. Then, the kernel is generated by the divisor being linearized;
see \cite[Corollary 2.6]{TomariWatanabeNormalZrGradedRings}, and so it
is finite. 

Our argument has the salient feature of showing that, in order to prove that $\Cl R$ is finitely generated for strongly $F$-regular singularities, one may assume that $\Cl R$ is torsion-free.
\end{remark}

\begin{corollary}[{\cf \cite[Corollary 3.7]{CarvajalSchwedeTuckerEtaleFundFsignature}}] \label{cor.DivisorClassGroup1}
Let $\sA$ be an ample invertible sheaf on a globally $F$-regular projective variety $Y$  and let $A\coloneqq\bigoplus_{i\geq 0} H^0(Y, \sA^i)$ be its section ring. If $\sA = \sL ^n$ for another invertible sheaf $\sL$ then $n\leq  1/s(A_O)$, where $O \in \Spec A\eqqcolon C(Y)$ is the vertex of the cone.
\end{corollary}
\begin{proof}
Recall that $Y$ is globally $F$-regular if and only if $A$ is strongly $F$-regular \cite{SchwedeSmithLogFanoVsGloballyFRegular}. Thus, $s(A_O)>0$. Let $B\coloneqq \bigoplus_{i\geq 0} H^0(Y, \sL^i)$. Then, $A$ is the $n^\mathrm{th}$ Veronese subring of $B$ and so $A_{O}^{\textnormal{sh}} \subset B_{O}^{\textnormal{sh}}$ is a degree $n$ Veronese-type cyclic cover. Then, $1 \geq s(B_{O}^{\textnormal{sh}}) = n \cdot s(A_O^{\textnormal{sh}}) = n \cdot s(A)$.
\end{proof}
\begin{corollary} \label{cor.DivisorClassGroup}
Let $Y$ be a globally $F$-regular projective variety. Then, $\Pic Y$ is torsion-free.
\end{corollary}
\begin{proof}
Let $D$ be a Weil divisor on $Y$ of index $n$ and write $\divisor_Y a + n D
=0$ for some $a \in K(Y)^{\times}$. Let $f\mathrel{:} Y' = \mathbf{Spec}_Y
\bigoplus_{i=0}^{n-1} \sO_Y(iD) \to Y$ be the corresponding cyclic
cover, which is a $\mu_n$-quasitorsor (in fact, a
$\mu_n$-torsor over the Cartier locus of $D$). Let $\sA$, $A$,
and $C(Y)$ be as in \autoref{cor.DivisorClassGroup1}. Recall that $\Cl Y$ is an extension of $\Cl C(Y)$ by $\Z$; see \cite[II, Exercise
6.3]{Hartshorne}. Thus, we may think of $D$ as a divisor on  $C(Y)$ of
index $n$ and the same formula $\divisor_{C(Y)} a + n D =0$ works with
$a \in K(Y)(t)^{\times}$ (where $K(Y)(t)$ is the function field of
$C(Y)$). Furthermore, $C(D,a;n)$ coincides with the section ring of
$Y'$ with respect to $f^* \sA$---both rings are the normalization of
$A$ inside $K(Y)(t,a^{1/n})$.\footnote{This can be verified by direct
  computation too using that $A(D) = \bigoplus_{i=0}^{\infty}
  H^0(Y,\sO_Y(D) \otimes\sA^i)$. One may think of this equality as the
  definition of $D$ as a divisor on $C(Y)$; see
  \cite[\S5.2]{SchwedeSmithLogFanoVsGloballyFRegular}.} Therefore,
$Y'$ is a globally $F$-regular projective variety. Now, it is
well-known that $\chi(Y,\sO_Y) = 1 = \chi(Y',\sO_{Y'})$;  see
\cite[Corollary 4.3]{SmithGloballyFRegular}. However, if $D$ is further Cartier, we have:
\[
  \chi(Y',\sO_{Y'}) = \chi(Y,f_*\sO_{Y'})
  = \chi\left(Y, \bigoplus_{i=0}^{n-1} \sO_Y(iD)\right)
  = \sum_{i=0}^{n-1} \chi\big(Y,\sO_Y(iD)\big)
  = \sum_{i=0}^{n-1}\chi(Y,\sO_Y)
  = n\chi(Y,\sO_Y),
\]
where we used that $\chi\big(Y,\sO_Y(E)\big) = \chi(Y,\sO_Y)$ for all
torsion Cartier divisors $E$ on $Y$, which follows from the
Hirzebruch--Riemann--Roch theorem. Indeed, 
\begin{align*}
  \chi\big(Y,\sO_Y(E)\big)
  &= \int_Y \chern(E) \cdot \td(Y)
    = \int_Y (1+E + \tfrac{1}{2}E^2 + \cdots ) \cdot \td(Y) \\
  &= \int_Y\td(Y) + \int_Y E\cdot ( 1+\tfrac{1}{2}E + \cdots )
    \cdot \td(Y) \\
  &= \chi(Y,\sO_Y) + \int_Y E\cdot ( 1+\tfrac{1}{2}E + \cdots )
    \cdot \td(Y).
\end{align*}
However, if $e \cdot E \sim 0$, then $e$ times the integer $\int_Y
E\cdot ( 1+\tfrac{1}{2}E + \cdots ) \cdot \td (Y) $ is zero.
\end{proof}

\begin{remark}
Globally $F$-regular varieties are of Fano type
\cite{SchwedeSmithLogFanoVsGloballyFRegular}. To the best of the
author's knowledge, the statement of \autoref{cor.DivisorClassGroup}
is open in general for varieties of Fano type ($p > 0$).
\end{remark}

\begin{remark}
Another proof of \autoref{cor.DivisorClassGroup} can be obtained as
follows. By \cite[Corollary 4.3]{SmithGloballyFRegular},
$\chi(Y,\sO_Y(N)) = h^0(Y,\sO_Y(N))$ for all nef Cartier divisors $N$ (\eg
torsion). Then, $0=\chi(Y,\sO_Y(D)) = \chi(Y,\sO_Y) = 1$ in case
$n>1$.  So we have $n=1$.
\end{remark}

The author is thankful to Fabio Bernasconi for showing him the following example and for kindly letting him to include it here. The example shows that we cannot replace $\Pic Y$ by $\Cl Y$ in \autoref{cor.DivisorClassGroup}, \emph{i.e.} there are globally $F$-regular varieties with torsion in their divisor class group. In fact, it shows that there are (necessarily non-factorial) toric varieties in any characteristic with torsion in their divisor class group.

\begin{example}
Set $X = \bP^1 \times \bP^1$ and denote the projection morphisms by $\pi_1, \pi_2  \mathrel{:} X \to \bP^1$. Here, we may define $\bP^1$ over an arbitrary algebraically closed field of any characteristic (including zero). Let us write $\bP^1 = \mathbb{G}_{\textnormal{m}} \cup \{0,\infty\}$ for the standard toric structure on $\bP^1$. For $i=1,2$, let us define $F_i \coloneqq \pi_i^* 0 \subset X$. Let $X_1 \to X$ be the blowup of $X$ at the intersection point $F_1 \cap F_2$ and let $E \subset X_1$ denote the corresponding exceptional curve. By abuse of notation, we shall denote the strict transform of $F_i$ in $X_1$ by $F_i$. Let $X_2 \to X_1$ be the blowup of $X_1$ at the intersection points $F_1 \cap E$ and $F_2 \cap E$, and let $D_1$, $D_2$ denote the corresponding exceptional curves. Observe that $X_1$ and $X_2$ are toric varieties, for we blew up points fixed by the torus action. Note that $E$ is a $(-3)$-curve on $X_2$ whereas $F_1$ and $F_2$ are $(-2)$-curves (once again, we have denoted strict transforms by the same letters). By contracting the curves $F_1$ and $F_2$ on $X_2$ (which are fixed by the torus action), we obtain a toric surface $Y$ with two singular points $y_1,y_2 \in Y$ corresponding respectively to the contractions of $F_1,F_2\subset X_2$ along $X_2 \to Y$. It is worth noting that these are $A_1$ points and so non-factorial. In fact, $\Cl \sO_{Y,y_i} \cong \Z/2$. For $i=1,2$, let $D'_i \subset Y$ be the push-forward of $D_i \subset X_2$ along the contraction $X_2 \to Y$. Observe that $D\coloneqq D_1'-D_2'$ has Cartier index $2$, for it generates the divisor class groups of $\sO_{Y,y_1}$ and $\sO_{Y,y_1 }$. In particular, $D$ is not linearly equivalent to zero. However, by the symmetry in the construction, one sees that $D$ is numerically trivial. Thus, $2D$ is a Cartier divisor numerically equivalent to zero. Since $H^1(Y,\sO_Y)=0$ and $\Pic Y$ is torsion free, numerical equivalence coincides with linear equivalence on Cartier divisors on $Y$ (use \cite[Corollary 9.5.13, Theorem 9.6.3]{KleimanPicardScheme}, \cf \cite[Corollary 1.4.38]{LazarsfeldPositivityI}). Hence, $2D$ is linearly equivalent to zero. In conclusion, $D$ defines a torsion element of order $2$ in $\Cl Y$. One also sees that $\chi(Y,\sO_Y(D))=h^0(Y,\sO_Y(D))=0 \neq 1 = \chi(Y,\sO_Y)$ and so why the proof of \autoref{cor.DivisorClassGroup} is not valid for non-Cartier Weil divisors.
\end{example}

\begin{remark}
As we just saw above, the divisor class group of a globally $F$-regular variety $Y$ may have torsion. However, by using \autoref{cor.TorsionPicard} and the cone construction, one readily sees that the torsion of $\Cl Y$ is bounded by $1/s(A_O)$ (with notation as in \autoref{cor.DivisorClassGroup1}). Indeed, since $\Cl Y$ is an extension of $\Cl C(Y)$ by $\Z$, the order of a torsion element of $\Cl Y$ is the same as the order of its image in the quotient $\Cl C(Y)$. 
\end{remark}



\begin{thebibliography}{XXXXXXXX+}
\bibitem[AL03]{AberbachLeuschke}
  I. M. Aberbach and G. J. Leuschke,
  {\em The $F$-signature and strong $F$-regularity},
  Math. Res. Lett. \textbf{10} (2003), no. 1, 51--56.

\bibitem[AK70]{AltmanKleimanIntroToGrothendieckDuality}
  A. Altman and S. Kleiman,
  {\em Introduction to Grothendieck duality theory},
  Lecture Notes in Mathematics, 146,
  Springer-Verlag, Berlin-New York, 1970.

\bibitem[Art75]{ArtinWildlyRamifiedZ2Actions}
  M. Artin,
  {\em Wildly ramified $\Z/2$ actions in dimension two}, Proc. Am. Math. Soc. \textbf{52} (1975), 60--64.

\bibitem[BGO17]{BhattGabberOlssonSpreadOut}
  B. Bhatt, O. Gabber, and M. Olsson,
  {\em Finiteness of \'etale fundamental groups by reduction modulo $p$},
preprint \arXiv{1705.07303} (2017).

\bibitem[Bha12]{BhattAnnihilatingCohomologyGroupSchemes}
  B. Bhatt,
  {\em Annihilating the cohomology of group schemes},
  Algebra Number Theory \textbf{6} (2012), no. 7, 1561--1577.

\bibitem[BC-RGST19]{BhattCarvajalRojasGrafSchwedeTucker}
  B. Bhatt, J. Carvajal-Rojas, P. Graf, K. Schwede, and K. Tucker,
  {\em \'Etale fundamental groups of strongly $F$-regular schemes},
  Int. Math. Res. Not. (2019), no.~14, 4325--4339.

\bibitem[BST12]{BlickleSchwedeTuckerFSigPairs1}
  M. Blickle, K. Schwede, and K. Tucker,
  {\em $F$-signature of pairs and the asymptotic behavior of Frobenius
  splittings},
  Adv. Math. \textbf{231} (2012), no. 6, 3232--3258.

\bibitem[BW88]{BlockWilsonClassification}
  R. E. Block and R. L. Wilson,
  {\em Classification of the restricted simple Lie algebras},
  J. Algebra \textbf{114} (1988), no. 1, 115--259.

\bibitem[Bou78]{BoutotSchemaPicardLocal}
  J.-F. Boutot,
  {\em Sch\'ema de Picard local},
  Lecture Notes in Mathematics, 632,
  Springer, Berlin, 1978.

\bibitem[Bri17]{BrionSomeStructureTheoremsForAlgebraicGroups}
  M. Brion,
  {\em Some structure theorems for algebraic groups}.
  In: {\em Algebraic groups: structure and actions},
  Proc. Sympos. Pure Math., vol. 94, Amer. Math. Soc.,
  Providence, RI, 2017, pp. 53--126.

\bibitem[C-RST18]{CarvajalSchwedeTuckerEtaleFundFsignature}
  J. Carvajal-Rojas, K. Schwede, and K. Tucker,
  {\em Fundamental groups of $F$-regular singularities via
    $F$-signature}, 
  Ann. Sci. \'Ec. Norm. Sup\'er. (4) \textbf{51} (2018), no. 4, 993--1016.

\bibitem[CEPT96]{ChinburgErezPappasTaylorTameActions}
  T. Chinburg, B. Erez, G. Pappas, and M. J. Taylor,
  {\em Tame actions of group schemes: integrals and slices},
  Duke Math. J. \textbf{82} (1996), no. 2, 269--308.

\bibitem[CF92]{CohenFischmanSemisimpleExtensions}
  M. Cohen and D. Fischman,
  {\em Semisimple extensions and elements of trace 1},
  J. Algebra \textbf{149} (1992), no. 2, 419--437.

\bibitem[Deb77]{DebremaekerNonAbelianCohomology}
  R. Debremaeker,
  {\em Non abelian cohomology},
  Bull. Soc. Math. Belg. \textbf{29} (1977), no. 1, 57--72.

\bibitem[DG70]{DemazureGrabielGroupes}
  M. Demazure and P. Gabriel,
  {\em Groupes alg\'ebriques.
  Tome I: G\'eom\'etrie alg\'ebrique, g\'en\'eralit\'es, groupes commutatifs},
  Masson \& Cie, \'Editeur, Paris;
  North-Holland Publishing Co., Amsterdam, 1970.

\bibitem[Doi85]{DoiAlgebraswithTotalIntegrals}
  Y. Doi,
  {\em Algebras with total integrals},
  Commun. Algebra \textbf{13} (1985), no. 10, 2137--2159.

\bibitem[EvdGM18]{AbelianVarietiesGeerMooen}
  B. Edixhoven, G. van der Geer, and B. Moonen,
  {\em Abelian varieties}, 2018. Available from \url{http://gerard.vdgeer.net/AV.pdf}.

\bibitem[EV10]{EsnaultViehwegSurfaceSingularitiesDominatedSmoothVarieties}
  H. Esnault and E. Viehweg,
  {\em Surface singularities dominated by smooth varieties},
  J. Reine Angew. Math.\textbf{649} (2010),
  1--9.

\bibitem[Fle75]{FlennerReineLokaleRinge}
  H. Flenner,
  {\em Reine lokale Ringe der Dimension zwei},
  Math. Ann. \textbf{216} (1975), no. 3, 253--263.

\bibitem[Gir71]{GiraudNonabelianCohomology}
  J. Giraud,
  {\em Cohomologie non ab\'elienne},
  Springer-Verlag, Berlin-New York, 1971,
  Die Grundlehren der mathematischen Wissenschaften, Band 179.

\bibitem[GJRW96]{GuralnickJaffeRaskindWiegandOnThePicardGroup}
  R. Guralnick, D. B. Jaffe, W. Raskind, and R. Wiegand,
  {\em On the Picard group: torsion and the kernel induced by a
    faithfully flat map},
  J. Algebra \textbf{183} (1996), no. 2, 420--455.

\bibitem[Har77]{Hartshorne}
  R. Hartshorne,
  {\em Algebraic geometry},
  Graduate Texts in Mathematics, vol. 52,
  Springer-Verlag, New York, 1977.

\bibitem[Har94]{HartshorneGeneralizedDivisorsOnGorensteinSchemes}
  \bysame,
  {\em Generalized divisors on Gorenstein schemes}.
  In: {\em Proceedings of Conference on Algebraic Geometry and Ring Theory
  in honor of Michael Artin}, Part III
  (Antwerp, 1992), vol. 8(3), 1994, pp. 287--339.

\bibitem[Hoc73]{HochsterContractedIdealsFromIntegralExtensions}
  M. Hochster,
  {\em Contracted ideals from integral extensions of regular rings},
  Nagoya Math. J. \textbf{51} (1973), 25--43.

\bibitem[HH89]{HochsterHunekeTightClosureAndStrongFRegularity}
  M. Hochster and C. Huneke,
  {\em Tight closure and strong $F$-regularity},
  M\'em. Soc. Math. France (N.S.) (1989), no. 38, 119--133,
  Colloque en l'honneur de Pierre Samuel (Orsay, 1987).

\bibitem[HL02]{HunekeLeuschkeTwoTheoremsAboutMaximal}
  C. Huneke and G. J. Leuschke,
  {\em Two theorems about maximal Cohen-Macaulay modules},
  Math. Ann. \textbf{324} (2002), no. 2, 391--404.

\bibitem[Kaw99]{KawamataIndex1CoversOfLogTerminalSurfaceSingularities}
  Y. Kawamata,
  {\em Index 1 covers of log terminal surface singularities},
  J. Algebraic Geom. \textbf{8} (1999), no.~3, 519--527.

\bibitem[KS10]{KerzSchmidtOnDifferentNotionsOfTameness}
  M. Kerz and A. Schmidt,
  {\em On different notions of tameness in arithmetic geometry},
  Math. Ann. \textbf{346} (2010), no. 3, 641--668.

\bibitem[Kle05]{KleimanPicardScheme}
  S. L. Kleiman,
  {\em The Picard scheme},
  in: {\em Fundamental algebraic geometry}, pp.~235--321,
  Math. Surveys Monogr. 123, Amer. Math. Soc., Providence, RI, 2005.

\bibitem[KT81]{KreimerTakeuchi}
  H. F. Kreimer and M. Takeuchi,
  {\em Hopf algebras and Galois extensions of an algebra},
  Indiana Univ. Math. J. \textbf{30} (1981),
  no. 5, 675--692.

\bibitem[Kun69]{KunzCharacterizationsOfRegularLocalRings}
  E. Kunz,
  {\em Characterizations of regular local rings for characteristic p},
  Am. J. Math. \textbf{91} (1969), 772--784.

\bibitem[LS69]{LarsonSweedlerTrace}
  R. G. Larson and M. E. Sweedler,
  {\em An associative orthogonal bilinear form for Hopf algebras},
  Am. J. Math. \textbf{91} (1969), 75--94.

\bibitem[Laz04]{LazarsfeldPositivityI}
    R. Lazarsfeld,
    {\em Positivity in algebraic geometry. {I}},
    Ergebnisse der Mathematik und ihrer Grenzgebiete. 3. Folge. A Series of Modern Surveys in Mathematics, vol. 48, Springer-Verlag, Berlin, 2004.

\bibitem[LLX20]{ChiLiuChenyangNormalizedVolumes}
  C. Li, Y. Liu, and C. Xu,
  {\em A guided tour to normalized volume}, in: {\em Geometric analysis},
  Cham: Birkh\"auser, Prog. Math. 333, 167-219 (2020).


\bibitem[LMM21a]{LiedtkeMartinMatsumotoLinearlyreductiveQuotientSingularities}
C. Liedtke, G. Martin, and Y. Matsumoto,
  {\em Linearly Reductive Quotient Singularities},
  preprint \arXiv{2102.01067} (2021).
  
\bibitem[LMM21b]{LiedtkeMartin}
  \bysame,
  {\em Torsors over the rational double points in characteristic p},
  preprint \arXiv{2110.03650} (2021)


\bibitem[Ma88]{MaFSplittings}
  F. Ma,
  {\em Splitting in integral extensions},
  Cohen-Macaulay modules and algebras,
  J. Algebra \textbf{116} (1988), no. 1, 176--195.

\bibitem[Mar16]{MarramaPurityForTorsors}
  A. Marrama,
  {\em A purity theorem for torsors},
  ALGANT Master Thesis, Universiteit Leiden \& Universit\"at Duisburg-Essen 2016. Available from \url{https://perso.pages.math.cnrs.fr/users/andrea.marrama/}.

\bibitem[Mar21]{MArtinBoundOnClRforSFR}
  I. Martin,
  {\em The number of torsion divisors in a strongly $F$-regular ring is
  bounded by the reciprocal of $F$-signature},
  Commun. Algebra (2021).
  Published online \href{https://doi.org/10.1080/00927872.2021.1986057}{doi:10.1080/00927872.2021.1986057}.

\bibitem[Mil17]{MilneAlgebraicGroups}
  J. S. Milne,
  {\em Algebraic groups. The theory of group schemes of finite type over a field}, Cambridge Studies in Advanced Mathematics, Cambridge University Press, 2017.

\bibitem[Mil80]{MilneEtaleCohomology}
  \bysame,
  {\em \'Etale cohomology},
  Princeton Mathematical Series, vol. 33, Princeton University
  Press, Princeton, NJ, 1980.

\bibitem[Mon93]{MontgomeryHopfAlgebras}
  S. Montgomery,
  {\em Hopf algebras and their actions on rings},
  CBMS Regional Conference Series
  in Mathematics, vol. 82, American Mathematical Society,
  Providence, RI, 1993.

\bibitem[M-B85]{MoretBaillyPurete}
  L. Moret-Bailly,
  {\em Un th\'eor\`eme de puret\'e pour les familles de courbes lisses},
  C. R. Acad. Sci. Paris S\'er. I Math. \textbf{300} (1985),
  no. 14, 489--492.

\bibitem[Mum61]{MumfordFundGroup}
  D. Mumford,
  {\em The topology of normal singularities of an algebraic surface
    and a criterion for simplicity},
  Inst. Hautes \'Etudes Sci. Publ. Math. (1961), no. 9, 5--22.

\bibitem[Mum76]{MumfordAlgebraicGeometryI}
  \bysame,
  {\em Algebraic geometry. I},
  Springer-Verlag, Berlin-New York, 1976,
  Grundlehren der Mathematischen Wissenschaften, No. 221.

\bibitem[Mum08]{MumfordAbelianVarieties}
  \bysame,
  {\em Abelian varieties},
  Tata Institute of Fundamental Research Studies in Mathematics,
  vol.~5,
  Hindustan Book Agency, New Delhi, 2008.

\bibitem[Mur67]{MurreLecturesFundamentalGroups}
  J. P. Murre,
  {\em Lectures on an introduction to Grothendieck's theory of the
    fundamental group},
  Tata Institute of Fundamental Research, Bombay, 1967,
  Notes by S. Anantharaman,
  Tata Institute of Fundamental Research Lectures on Mathematics,
  No~40.

\bibitem[Nag62]{NagataCompleteReducibility}
  M. Nagata,
  {\em Complete reducibility of rational representations of a matrix
    group},
  J. Math. Kyoto Univ. \textbf{1} (1961/1962),
  87--99.

\bibitem[Nor82]{NoriFundamentalGroupScheme}
  M. V. Nori,
  {\em The fundamental group-scheme}.
  Proc. Indian Acad. Sci. Math. Sci. \textbf{91}
  (1982), no. 2, 73--122.

\bibitem[Pol22]{PolstraATheoremAboutMCMM}
  T. Polstra,
  {\em A theorem about maximal Cohen-Macaulay modules},
  Int. Math. Res. Not. IMRN 2022, no.~3, 2086--2094.

\bibitem[PT18]{TuckerPolstraUniformApproach}
  T. Polstra and K. Tucker,
  {\em $F$-signature and Hilbert-Kunz multiplicity,
    a combined approach and comparison},
  Algebra Number Theory \textbf{12} (2018), no. 1, 61--97.

\bibitem[PS08]{PremetStradeCompletion}
  A. Premet and H. Strade,
  {\em Simple Lie algebras of small characteristic. VI. Completion of
    the classification},
  J. Algebra \textbf{320} (2008), no. 9, 3559--3604.

\bibitem[SS10]{SchwedeSmithLogFanoVsGloballyFRegular}
  K. Schwede and K. E. Smith,
  {\em Globally $F$-regular and log Fano varieties},
  Adv. Math. \textbf{224} (2010), no. 3, 863--894.

\bibitem[ST12]{SchwedeTuckerTestIdealSurvey}
  K. Schwede and K. Tucker,
  {\em A survey of test ideals},
  in: {\em Progress in Commutative Algebra 2.
  Closures, Finiteness and Factorization}, Walter de Gruyter GmbH \& Co.
  KG, Berlin, 2012, pp. 39--99.

\bibitem[ST14]{SchwedeTuckerTestIdealFiniteMaps}
  \bysame,
  {\em On the behavior of test ideals under finite morphisms},
  J. Algebraic Geom. \textbf{23} (2014), no. 3, 399--443.

\bibitem[Sin03]{SinghCyclicCoversOfRational}
  A. K. Singh,
  {\em Cyclic covers of rings with rational singularities},
  Transactions of the American Mathematical Society \textbf{355}
  (2003), no. 3, 1009--1024.

\bibitem[Smi97]{SmithFRatImpliesRat}
  K. E. Smith,
  {\em $F$-rational rings have rational singularities},
  Am. J. Math. \textbf{119} (1997), no. 1, 159--180.

\bibitem[Smi00]{SmithGloballyFRegular}
  \bysame,
  {\em Globally $F$-regular varieties, applications to vanishing
    theorems for quotients of Fano varieties},
  Mich. Math. J. \textbf{48} (2000), 553--572.

\bibitem[SdB97]{SmithVanDenBerghSimplicityOfDiff}
  K. E. Smith and M. V. den Bergh,
  {\em Simplicity of rings of differential operators in prime
    characteristic},
  Proc. Lond. Math. Soc. (3) \textbf{75} (1997), no. 1, 32--62.

\bibitem[SZ15]{SmithZhang}
  K. E. Smith and W. Zhang,
  {\em Frobenius splitting in commutative algebra},
  in: {\em Commutative algebra and noncommutative algebraic geometry.
  Vol. I: Expository articles.}, Math. Sci. Res. Inst. Publ., vol.~67,
  Cambridge Univ. Press, New York, 2015, pp. 291--345.

\bibitem[Spe20]{SpeyerFrobeniusSplit}
  D. E. Speyer,
  {\em Frobenius split subvarieties pull back in almost all
    characteristics},
  J. Commut. Algebra \textbf{12} (2020), no. 4, 573--579.

\bibitem[SW91]{StradeWilsonClassification}
  H. Strade and R. L. Wilson,
  {\em Classification of simple Lie algebras over algebraically closed
    fields of prime characteristic},
  Bull. Am. Math. Soc., New Ser. \textbf{24} (1991), no. 2, 357--362.

\bibitem[Tat97]{TateFiniteFlatGroupSchemes}
  J. Tate,
  {\em Finite flat group schemes, Modular forms and Fermat's last
    theorem},
  (Boston, MA, 1995), Springer, New York, 1997, pp. 121--154.

\bibitem[Stacks]{stacks-project}
  The Stacks project authors,
  {\em The stacks project}, \url{http://stacks.math.columbia.edu} (2022).

\bibitem[TW92]{TomariWatanabeNormalZrGradedRings}
  M. Tomari and K. Watanabe,
  {\em Normal $\mathbb{Z}_r$-graded rings and normal cyclic covers},
  Manuscr. Math. \textbf{76} (1992), no. 3-4, 325--340.

\bibitem[Tuc12]{TuckerFSigExists}
  K. Tucker,
  {\em $F$-signature exists},
  Invent. Math. \textbf{190} (2012), no. 3, 743--765.

\bibitem[Viv10]{VivianiSimpleFIniteGroupSchemes}
  F. Viviani,
  {\em Simple finite group schemes and their infinitesimal
    deformations},
  Rend. Semin. Mat. Univ. Politec. Torino \textbf{68} (2010),
  no. 2, 171--182.

\bibitem[Von12]{VonKorffFSigOfAffineToric}
  M. R. Von Korff,
  {\em The $F$-signature of toric varieties},
  Ph.D. thesis, University of Michigan, Ann Arbor, MI, 2012. Available from \url{https://www.proquest.com/docview/1149994294}.

\bibitem[Wat91]{WatanabeFRegularFPure}
  K. Watanabe,
  {\em $F$-regular and $F$-pure normal graded rings},
  J. Pure Appl. Algebra \textbf{71} (1991), no. 2-3, 341--350.

\bibitem[Xu14]{XuFinitenessOfFundGroups}
  C. Xu,
  {\em Finiteness of algebraic fundamental groups},
  Compos. Math. \textbf{150} (2014), no. 3, 409--414.
\end{thebibliography}
\end{document}